\documentclass[10pt]{article}
\usepackage{amsmath, amsfonts, amsthm, amssymb, amsrefs} 
\usepackage{mathtools}
\usepackage{tikz, graphicx, pgfplots, color, xcolor}
	\pgfplotsset{compat=1.12} 
	\usetikzlibrary{positioning}	
	\newcommand{\tikzc}[1]{\begin{center}\begin{tikzpicture}#1\end{tikzpicture}\end{center}}
\usepackage{mhchem, chemfig}
\usepackage[left=1.2in,right=1.2in,top=1in,foot=1in]{geometry}
\allowdisplaybreaks	
\usepackage[linkbordercolor=cyan, pdfborder={0 0 0.5}]{hyperref}	 
\usepackage{enumitem}

	\newcommand\blue[1]{{\textcolor{blue}{#1}}}
	\definecolor{orange}{RGB}{250, 140, 0}
		
	\definecolor{turq}{RGB}{0, 160, 160}
		
\usepackage{amsthm}
	\newtheorem{thm}{Theorem}[section]
	\newtheorem{lem}[thm]{Lemma}
	\newtheorem{prop}[thm]{Proposition}
	
	\theoremstyle{definition}
		\newtheorem{defn}[thm]{Definition}
		\newtheorem{ex}[thm]{Example}
	\newtheoremstyle{TheoremNum}
        {\topsep}{\topsep}              
        {\itshape}                      
        {}                              
        {\bfseries}                     
        {.}                             
        { }                             
        {\thmname{#1}\thmnote{ \bfseries #3}}
    \theoremstyle{TheoremNum}
	
	\newcommand{\rmk}{\noindent {\bf Remark.\quad}}
	
	\newcommand{\df}[1]{{\bf\emph{#1}}}		
\newcommand{\eq}[1]{\begin{align*}#1\end{align*}}
	\newcommand{\eqn}[1]{\begin{align}#1\end{align}}  
	
\newcommand{\ds}{\displaystyle}			
\renewcommand{\emptyset}{\varnothing}	
	\renewcommand{\epsilon}{\varepsilon}	
	\renewcommand{\phi}{\varphi}			
\renewcommand{\tilde}[1]{\widetilde{#1}} 

\newcommand{\rr}{\ensuremath{\mathbb{R}}}

\DeclareMathOperator{\Ker}{ker}			
\DeclareMathOperator{\Img}{im}			
\DeclareMathOperator{\Span}{span}		

\newcommand{\kk}{k}

\newcommand{\rrpp}{\rr_{>0}}
\newcommand{\RR}{\ensuremath{\rightleftharpoons}}
\newcommand{\FR}{\ensuremath{\rightarrow}}

\newcommand{\vv}[1]{{\boldsymbol{#1}}}
\newcommand{\cf}[1]{\chemfig{#1}}

\title{An efficient characterization of complex-balanced, detailed-balanced, and weakly reversible systems}

\author{
   Gheorghe Craciun\thanks{Department of Mathematics and Department of Biomolecular Chemistry, University of Wisconsin-Madison, Madison, WI, 53706 (craciun@math.wisc.edu).} 
\and
    Jiaxin Jin\thanks{Department of Mathematics, University of Wisconsin-Madison, Madison, WI, 53706 (jjin43@wisc.edu).
    } 
\and
    Polly Y. Yu\thanks{Department of Mathematics, University of Wisconsin-Madison, Madison, WI, 53706 (pollyyu@math.wisc.edu).} 
}

\begin{document}
\maketitle

\begin{abstract}
Very often, models in biology, chemistry, physics, and engineering are systems of polynomial or power-law ordinary differential equations, arising from a reaction network. Such dynamical systems can be generated by many different reaction networks. On the other hand, networks with special properties (such as reversibility or weak reversibility) are known or conjectured to give rise to dynamical systems that have special properties: existence of positive steady states, persistence, permanence, and (for well-chosen parameters) complex balancing or detailed balancing. These last two are related to thermodynamic equilibrium, and therefore the positive steady states are unique and stable. We describe a computationally efficient characterization of polynomial or power-law dynamical systems that can be obtained as complex-balanced, detailed-balanced, weakly reversible, and reversible mass-action systems.
\end{abstract}

\section{Introduction}
\label{sec:Intro}

Many mathematical models in biology, chemistry, physics, and engineering are obtained from nonlinear interactions between several species or populations, such as (bio)chemical reactions in a cell or a chemical reactor, population dynamics in an ecosystem, or kinetic interactions in a gas or solution~\cite{GunaNts, HJ72, CasianGAC1, TDS, ctf06, CraciunBinh, Banaji_2013, Fein72,  FeinLectNts, Feinberg_1987, Savageau_Voit, Sontag1, angeli_tutorial}. Very often, these models are generated by a graph of interactions according to specific kinetic rules; \emph{mass-action kinetics} for reaction network models is one such example~\cite{RevCY}. 

If the graph underlying the mass-action system in a given reaction network has some special properties, then the associated dynamical system is known (or conjectured) to have certain dynamical properties. For example,  dynamical systems generated by \emph{reversible} reaction networks are known to have at least one positive steady state within each linear invariant subspace~\cite{Boros_2018}. Moreover, these models are known to be persistent and permanent if the number of species is small and are conjectured to have these properties for any number of species~\cite{CasianGAC2, CasianGAC1}. The same situation occurs for \emph{weakly reversible} reaction networks, i.e., for networks where each reaction is part of a cycle (see Figure~\ref{fig:excrnrev}(b) and (c) for examples of such networks). For descriptions of other important classes of networks, see~\cite{ABCJ_2018}.

Moreover, after some restrictions on the parameter values, weakly reversible networks give rise to \emph{complex-balanced systems}, which are known to have a unique locally stable steady state within each linear invariant subspace. This steady state is known to be globally stable under some additional assumptions~\cite{CasianGAC1, DaveGAC, geomGAC, CasianGAC2} and is actually conjectured to be globally stable even without these assumptions~\cite{CasianGAC1, CraciunGAC}. If a reaction network is a complex-balanced system under mass-action kinetics, then other relevant models, ranging from continuous-time Markov chain models~\cite{Anderson_Craciun_Kurtz} to reaction-diffusion models~\cite{Desvillettes_2017, method_of_lines} and delay differential equation models~\cite{delayCB}, are also stable in some sense.

It turns out that the same dynamical system can be generated by a multitude of reaction networks~\cite{cpIdentifiability, Sze2011, WRAlgDense, WRAlgSparse,HarsToth1987}. Therefore, if a system is generated by a network that does \emph{not} enjoy a specific graphical property (e.g., not weakly reversible), we can ask whether the same system \emph{may} be generated by a weakly reversible network. Others have asked this question before and formulated algorithms for a given number of complexes~\cite{PolylTimeAlgWR, Sze2011, WRAlgDense, WRAlgSparse} and applied the results to designing control systems~\cite{Sontag1, DEMASControl}. In order to determine whether a given system is generated by a weakly reversible or complex-balanced system, one would have to determine if it can be done using $n$ number of complexes for all $n \geq 1$. 

In this paper we develop a theory of dynamical equivalence between mass-action systems (or more generally, polynomial or power-law dynamical systems) and weakly reversible and complex-balanced systems. Our results allow us to reformulate this dynamical equivalence problem as a linear feasibility problem whose dimension depends only on the size of the original system.  

In order to describe our main results, we need to introduce some definitions and notations (these notions will be described in further detail in Section 2). For our purposes here, a \emph{reaction network} is an oriented graph $G = (V_G,E_G)$ with vertex set $V_G$ and edge set $E_G$ such that $V_G \subseteq \rr^n$. If $\vv y$, $\vv y'\in V_G$ and $(\vv y, \vv y')$ is an edge in $E_G \subseteq V_G\times V_G$, then we write $\vv y\to \vv y'\in G$. With these notations, a dynamical system generated by $G$ (according to mass-action kinetics) is a system of ordinary differential equations on $\rrpp^n$ given by 
	\eqn{
	\label{eq:eqMAS_1}
		\frac{d\vv x}{dt} 
		\,\,=\!\!
		\sum_{\vv y\FR \vv y' \in G} \kk_{\vv y\FR \vv y'} \vv x^{\vv y} (\vv y' - \vv y),
	}
where $\vv x \in \rrpp^n$,  $\vv x^{\vv y} = x_1^{y_1} x_2^{y_2} \cdots x_n^{y_n}$, and $\kk_{\vv y\FR \vv y'} > 0$ for all $\vv y\FR \vv y' \in G$. We will denote the dynamical system~(\ref{eq:eqMAS_1}) by $G_{\vv \kk}$, where $\vv \kk$ is the vector of parameters $\kk_{\vv y\FR \vv y'}$ for all $\vv y\FR \vv y' \in G$.

\bigskip

One of our main results is the following.
\newtheorem*{thm*}{Theorem}

\begin{thm*}
    A mass-action system $G_{\vv \kk}$ is dynamically equivalent to some \emph{complex-balanced} mass-action system if and only if it is dynamically equivalent to a complex-balanced mass-action system $G'_{\vv \kk'}$ that only uses the vertices of $G$, i.e., with $V_{G'} \subseteq V_G$.
\end{thm*}

This theorem is useful not only for finding complex-balanced realizations of mass-action systems but also because for the first time, it gives us {\em a computationally feasible way to decide if such realizations exist}, as we only need to check if they exist for graphs $G'$ that have $V_{G'} \subseteq V_{G}$. 

\bigskip

We will see in Section~\ref{sec:NoNewNodes} that we can restrict the set $V_{G'}$ even more: without loss of generality we can  assume that it is contained in the set of ``source vertices" of $G$. We have also obtained similar results for other important classes of mass-action systems: detailed-balanced, weakly reversible and reversible systems. 
Moreover, our results are shown for \emph{flux systems}, which allows for other types of kinetics beside mass-action kinetics (Section~\ref{sec:FluxesCRN}).

Reaction networks and mass-action systems, along with all other relevant terms, are defined in Section~\ref{sec:crn}. We view a reaction network as a directed graph embedded in Euclidean space. In Section~\ref{sec:FluxesCRN}, we define fluxes on a reaction network and relate them back to mass-action systems. Section~\ref{sec:NoNewNodes} contains our main results for complex-balanced realizations, weakly reversible realizations, detailed-balanced realizations and reversible realizations. We make a brief comment on the implication of our results on the network's \emph{deficiency}. Finally, we present the relevant feasibility problems in Section~\ref{sec:numerical}.

\section{Reaction Networks and Mass-Action Systems}
\label{sec:crn}

Chemical reaction networks appear at the intersection of biology, biochemistry, chemistry, engineering, and mathematics. Different notations are used in the literature; here we explain the notations used throughout this paper. Introductions to chemical reaction network theory can be found in~\cite{RevCY, GunaNts,FeinLectNts}. 

\begin{defn}
\label{def:crn}
	A {\df{reaction network}} (or simply a \df{network}) is a directed graph $G = (V_G,E_G)$ embedded in Euclidean space, with no self-loops, i.e., $V_G \subseteq \rr^n$ and $E_G \subseteq V_G \times V_G$ and $(\vv y,\vv y) \not\in E_G$ for any $\vv y \in V_G$. 
\end{defn}

\noindent 
When there is no ambiguity, we simply write $G = (V,E)$.\\

\noindent 
\rmk 
Vertices are points in $\rr^n$, so an edge $e \in E$ can be regarded as a bona fide vector in $\rr^n$. We denote an edge $e = (\vv y, \vv y')$ as $\vv y \FR \vv y'$, which is associated to a {\df{reaction vector}} $\vv y' - \vv y \in \rr^n$. We also write $\vv y \to \vv y' \in G$ instead of $\vv y \to \vv y' \in E$.  \\

The dimension $n$ of the ambient Euclidean space is the number of chemical species involved in the reaction network $G$. An edge in the set $E$ is called a \df{reaction}. A vertex in $V$ is also known as a \df{reaction complex}. 
The \df{source vertex} of a reaction $\vv y \to \vv y'$ is the vertex $\vv y$, while $\vv y'$ is the \df{product vertex}. Let $V_s \subseteq V$ denote the \df{set of source vertices}, i.e., the set of vertices that is the source of some reaction. 

The vector space spanned by the reaction vectors is the \df{stoichiometric subspace} $S = \Span_\rr \{ \vv y' - \vv y \colon \vv y \FR \vv y' \in G\}$. For any positive vector $\vv x_0 \in \rrpp^n$, the affine polytope $(\vv x_0 + S)_> = (\vv x_0 + S) \cap \rrpp$ is known as the \df{stoichiometric compatibility class} of $\vv x_0$. A reaction network $G$ is \df{reversible} if $\vv y' \FR \vv y \in G$ whenever $\vv y \FR \vv y' \in G$; for simplicity, we denote such a pair of reactions by $\vv y \RR \vv y'$. It is \df{weakly reversible} if every connected component of $G$ is strongly connected, i.e., every reaction $\vv y \FR \vv y' \in G$ is part of an oriented cycle. 

\begin{ex}
\label{ex:crn}
	Figure~\ref{fig:excrn} shows a reaction network $G$ in $\rr^2$ with $6$ vertices and $3$ reactions. The reactions are
	\eq{
		\vv y_1 \FR \vv z_1 
		&= \begin{pmatrix} 1 \\ 0 \end{pmatrix}
			\FR \begin{pmatrix} 2 \\ 0 \end{pmatrix}, 
		\qquad
		\vv y_2 \FR \vv z_2 
		= \begin{pmatrix} 1 \\ 1 \end{pmatrix}
			\FR \begin{pmatrix} 0 \\ 2 \end{pmatrix},
		\qquad
		\vv y_3 \FR \vv z_3 
		= \begin{pmatrix} 0 \\ 1 \end{pmatrix}
			\FR \begin{pmatrix} 0 \\ 0 \end{pmatrix}.	
	}
    The stoichiometric subspace, which is the linear span of the reaction vectors, is $\rr^2$. In particular, any stoichiometric compatibility class is all of $\rrpp^2$. The reaction network $G$ is neither reversible nor weakly reversible.
\end{ex} 
	\begin{figure}[h!]
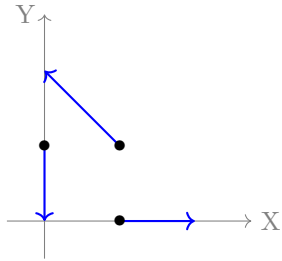

		\tikzc{
			\draw [->, gray] (-0.5,0)--(2.75,0) node [right] {X};
			\draw [->, gray] (0,-0.5)--(0,2.75) node [left] {Y};
			\draw [->, blue, thick] (1,0)--(2,0);
			\draw [->, blue, thick] (1,1)--(0,2);
			\draw [->, blue, thick] (0,1)--(0,0);
			\node (x) at (1,0) {$\bullet$};
			\node (xy) at (1,1) {$\bullet$};
			\node (y) at (0,1) {$\bullet$};
		}
		\caption{A reaction network $G$ in $\rr^2$ consisting of 3 reactions and 6 vertices. Under mass-action kinetics, this network gives rise to the classical Lotka--Volterra model for population dynamics.}
		\label{fig:excrn}
	\end{figure}
	
\begin{ex}
\label{ex:crnrev}
    Three more examples of reaction networks are presented in Figure~\ref{fig:excrnrev}. The reaction networks (a) $G$, (b) $G'$, and (c) $G^*$ share the vertices
	\eq{
		\vv y_1 = \begin{pmatrix} 0 \\ 0 \end{pmatrix}, \quad 
		\vv y_2 = \begin{pmatrix} 0 \\ 2 \end{pmatrix}, \quad 
		\vv y_3 = \begin{pmatrix} 3 \\ 2 \end{pmatrix}, 
		\quad \text{and} \quad 
		\vv y_4 = \begin{pmatrix} 3 \\ 0 \end{pmatrix}.
	}
    The reaction networks  $G$, $G^*$ have two additional vertices
    \eq{ 
        \vv y_5 = \begin{pmatrix}  1\\ 1 \end{pmatrix} 
        \quad \text{ and } \quad  
        \vv y_6 = \begin{pmatrix}  2\\ 1 \end{pmatrix}.
    }
    
    The set of four reactions of $G$ is $E_G = \{\vv y_1 \FR \vv y_5,  \, \vv y_2 \FR \vv y_5, \, \vv y_3 \FR \vv y_6 , \, \vv y_4 \FR \vv y_6\}$.  
    The set of reactions of $G'$ is $E_{G'} = \{\vv y_1 \RR \vv y_2, \, \vv y_2 \RR \vv y_3, \, \vv y_3 \RR \vv y_4, \, \vv y_4 \RR \vv y_1, \, \vv y_1 \RR \vv y_3, \, \vv y_2 \RR \vv y_4\}$. 
    The set of reactions of $G^*$ is $E_{G^*} = \{\vv y_1 \RR \vv y_5 \RR \vv y_2, \, \vv y_3 \RR \vv y_6 \RR \vv y_4, \,  \vv y_5 \RR \vv y_6, \,  \vv y_5 \FR \vv y_3 , \, \vv y_5 \FR \vv y_4\}$. 
    The networks $G'$ and $G^*$ are weakly reversible, and $G'$ is also reversible. The stoichiometric subspace is $S = \rr^2$ for all three networks. 

	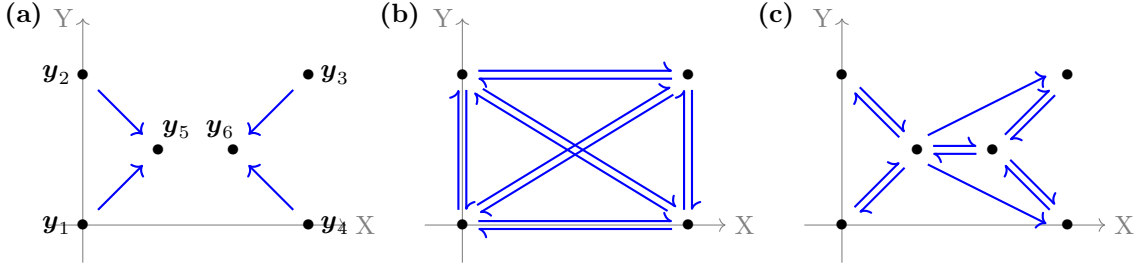
\begin{figure}[h!]
	\centering
	\vspace{0.15cm}
	\hspace{0.3cm}
		\begin{tikzpicture}
			\draw [->, gray] (-0.5,0)--(3.5,0) node [right] {X};
			\draw [->, gray] (0,-0.5)--(0,2.75) node [left] {Y};
			\node (1) at (0,0) {$\bullet$};
			\node (2) at (0,2) {$\bullet$};
			\node (3) at (3,2) {$\bullet$};
			\node (4) at (3,0) {$\bullet$};
			\node (5) at (1,1) {$\bullet$};
			\node (6) at (2,1) {$\bullet$};
			\draw [->, blue, thick] (1)--(5);
			\draw [->, blue, thick] (2)--(5);
			\draw [->, blue, thick] (3)--(6);
			\draw [->, blue, thick] (4)--(6);
			\node[left=-5pt of 1]    {$\vv y_1$};
			\node[left=-5pt of 2]    {$\vv y_2$};
			\node[right=-5pt of 3]    {$\vv y_3$};
			\node[right=-5pt of 4]    {$\vv y_4$};
			\node[above right =-5pt and -8pt of 5]    {$\vv y_5$};
			\node[above left =-5pt and -10pt  of 6]    {$\vv y_6$};
		\end{tikzpicture}
		\hspace{0.3cm}
		\begin{tikzpicture}
			\draw [->, gray] (-0.5,0)--(3.5,0) node [right] {X};
			\draw [->, gray] (0,-0.5)--(0,2.75) node [left] {Y};
			\node (1) at (0,0) {$\bullet$};
			\node (2) at (0,2) {$\bullet$};
			\node (3) at (3,2) {$\bullet$};
			\node (4) at (3,0) {$\bullet$};
            \draw [-{>[harpoon]}, blue, thick, transform canvas={xshift=0pt, yshift=1.5pt}] (2)--(3);
            \draw [-{>[harpoon]}, blue, thick, transform canvas={xshift=0pt, yshift=-1.5pt}] (3)--(2);
            \draw [-{>[harpoon]}, blue, thick, transform canvas={xshift=0pt, yshift=1.5pt}] (1)--(4);
            \draw [-{>[harpoon]}, blue, thick, transform canvas={xshift=0pt, yshift=-1.5pt}] (4)--(1);
            \draw [-{>[harpoon]}, blue, thick, transform canvas={yshift=0pt, xshift=1.5pt}] (2)--(1);
            \draw [-{>[harpoon]}, blue, thick, transform canvas={yshift=0pt, xshift=-1.5pt}] (1)--(2);
            \draw [-{>[harpoon]}, blue, thick, transform canvas={yshift=0pt, xshift=1.5pt}] (3)--(4);
            \draw [-{>[harpoon]}, blue, thick, transform canvas={yshift=0pt, xshift=-1.5pt}] (4)--(3);
            \draw [-{>[harpoon]}, blue, thick, transform canvas={xshift=-1pt, yshift=1pt}] (0.23,0.23)--(2.8,1.8);
            \draw [-{>[harpoon]}, blue, thick, transform canvas={xshift=1pt, yshift=-1pt}] (2.77,1.77)--(0.2,0.2);
            \draw [-{>[harpoon]}, blue, thick, transform canvas={xshift=1.2pt, yshift=1.2pt}] (0.23, 1.77)--(2.8, 0.2);
            \draw [-{>[harpoon]}, blue, thick, transform canvas={xshift=-1.2pt, yshift=-1.2pt}] (2.77,0.23)--(0.2, 1.8);
		\end{tikzpicture}
		\hspace{0.3cm}
		\begin{tikzpicture}
			\draw [->, gray] (-0.5,0)--(3.5,0) node [right] {X};
			\draw [->, gray] (0,-0.5)--(0,2.75) node [left] {Y};
			\node (1) at (0,0) {$\bullet$};
			\node (2) at (0,2) {$\bullet$};
			\node (3) at (3,2) {$\bullet$};
			\node (4) at (3,0) {$\bullet$};
			\node (5) at (1,1) {$\bullet$};
			\node (6) at (2,1) {$\bullet$};
			\draw [->, blue, thick, transform canvas={xshift=-2pt, yshift=2.5pt}]  (5) --(3);
			\draw [->, blue, thick, transform canvas={xshift=-2pt, yshift=-2.5pt}] (5)--(4);
			\draw [-{>[harpoon]}, blue, thick, transform canvas={xshift=0pt, yshift=1.5pt}] (5)--(6);
            \draw [-{>[harpoon]}, blue, thick, transform canvas={xshift=0pt, yshift=-1.5pt}] (6)--(5);
            \draw [-{>[harpoon]}, blue, thick, transform canvas={xshift=-1pt, yshift=1pt}] (1)--(5);
            \draw [-{>[harpoon]}, blue, thick, transform canvas={xshift=1pt, yshift=-1pt}] (5)--(1);
            \draw [-{>[harpoon]}, blue, thick, transform canvas={xshift=1pt, yshift=1pt}] (2)--(5);
            \draw [-{>[harpoon]}, blue, thick, transform canvas={xshift=-1pt, yshift=-1pt}] (5)--(2);
            \draw [-{>[harpoon]}, blue, thick, transform canvas={xshift=-1pt, yshift=0pt}] (6)--(3);
            \draw [-{>[harpoon]}, blue, thick, transform canvas={xshift=1pt, yshift=-2pt}] (3)--(6);
            \draw [-{>[harpoon]}, blue, thick, transform canvas={xshift=1pt, yshift=2pt}] (6)--(4);
            \draw [-{>[harpoon]}, blue, thick, transform canvas={xshift=-1pt, yshift=0pt}] (4)--(6);
		\end{tikzpicture}
		\begin{tikzpicture}[overlay]
			\node at (-15,3.3) {{\bf (a)}};
			\node at (-10,3.3) {{\bf (b)}};
			\node at (-5,3.3) {{\bf (c)}};
		\end{tikzpicture}
		\caption{Examples of reaction networks 
		(a) $G$, (b) $G'$, and (c) $G^*$, with labels of vertices shown in (a). The dynamical systems generated by the network (a) can also be generated by (b) or (c) for well-chosen rate constants. Note that (b) and (c) are weakly reversible, and (b) is also reversible.}
		\label{fig:excrnrev}
	\end{figure}
\end{ex}

A reaction network $G$ is associated to a dynamical system, by assuming that each reaction $\vv y \FR \vv y'$ proceeds according to a rate function $\nu_{\vv y\FR \vv y'}(\vv x)$, where $\vv x \in \rrpp^n$ is the vector of \emph{concentrations} of the chemical species in the system. One of the most extensively studied kinetic systems is \emph{mass-action kinetics}, where $\nu_{\vv y\FR \vv y'}(\vv x)$ is a monomial whose exponent vector is $\vv y$. 

\begin{defn}
\label{def:MAS}
	Let $G =(V,E)$ be a reaction network, and let $\vv \kk = (\kk_{\vv y\FR \vv y'})_{\vv y\FR \vv y' \in G} \in \rrpp^{E}$ be a vector of \emph{rate constants}. We call the weighted directed graph $G_{\vv \kk}$ a {\df{mass-action system}}, whose {\df{associated dynamical system}} is the system on $\rrpp^n$
	\eqn{
	\label{eq:eqMAS}
		\frac{d\vv x}{dt} 
		\,\,=\!\!
		\sum_{\vv y\FR \vv y' \in G} \kk_{\vv y\FR \vv y'} \vv x^{\vv y} (\vv y' - \vv y),
	}
    where $\vv x^{\vv y} = x_1^{y_1} x_2^{y_2} \cdots x_n^{y_n}$. By convention, $\vv x^{\vv 0} = 1$. 
\end{defn}

It is convenient to refer to $\kk_{\vv y \to \vv y'}$ even when $\vv y \to \vv y' \not\in G$, in which case we mean $\kk_{\vv y \to \vv y'} = 0$. We adopt the convention that the empty sum is $\vv 0$, i.e., $\sum_{\vv y \FR\vv  y' \in \emptyset} \kk_{\vv y \to \vv y'} (\vv y' - \vv y) = \vv 0$.\\

\begin{ex}
We revisit Example~\ref{ex:crn} under the assumption of mass-action kinetics. The dynamical system associated to this reaction network $G = (V,E)$ for an arbitrary vector of rate constants $\vv \kk = (k_j)_{\vv y_j \FR \vv z_j \in G} \in \rrpp^{E}$ is 
	\eq{
		\frac{d\vv x}{dt} 
		&= \kk_1 x \begin{pmatrix} 1 \\ 0 \end{pmatrix}
		+ \kk_2 xy \begin{pmatrix} -1 \\ 1 \end{pmatrix}
		+ \kk_3 y \begin{pmatrix} 0 \\ -1 \end{pmatrix} 
		= 
			\begin{pmatrix}
				\kk_1 x\hphantom{y} - \kk_2 xy \\
				\kk_2 xy - \kk_3 y \hphantom{x}
			\end{pmatrix}.
	}
This is the Lotka--Volterra population dynamics model.\\
\end{ex}

Given a mass-action system $G_{\vv \kk}$, (\ref{eq:eqMAS}) uniquely defines its associated dynamical system; however, many different reaction networks can give rise to the same dynamical system under mass-action kinetics. It has been known for a long time that if a reaction network has some special properties (e.g., reversible, weakly reversible, deficiency zero), then the mass-action system is known to have certain dynamical properties (e.g., existence of positive steady state, local and global stability). Therefore, given a mass-action system, we are interested in networks with richer structural properties that give rise to same dynamical systems. If two mass-action systems give rise to the same associated dynamical systems, we say they are \emph{dynamically equivalent}~\cite{Sze2011, WRAlgDense, WRAlgSparse, cpIdentifiability}.

\begin{defn}
\label{def:DE}
	Two mass-action systems $G_{\vv \kk}$ and $G'_{\vv \kk'}$ are {\df{dynamically equivalent}} if 
	\eqn{
	\label{eq:eqDE}
		\sum_{\vv y_1\FR \vv y_2 \in G} \kk_{\vv y_1 \FR \vv y_2} \vv x^{\vv y_1} (\vv y_2 - \vv y_1) 
		\,\,\,=\!\!
		\sum_{\vv y_1'\FR \vv y_2' \in G'} \kk'_{\vv y'_1 \FR \vv y'_2} \vv x^{\vv y'_1} (\vv y'_2 - \vv y'_1)
	}
    for all $\vv x \in \rrpp^n$. We say that $G'_{\vv\kk'}$ is another \df{realization} of $G_{\vv\kk}$. 
\end{defn}

\rmk From (\ref{eq:eqDE}), a necessary and sufficient condition for dynamical equivalence is 
\eqn{
\label{eq:DE}
	\sum_{\vv y_0\FR \vv y \in G} \kk_{\vv y_0 \FR \vv y} (\vv y - \vv y_0) 
	\,\,\,=\!\!
	\sum_{\vv y_0\FR \vv y' \in G'} \kk'_{\vv y_0 \FR \vv y'}  (\vv y' - \vv y_0)
}
for all $\vv y_0 \in V_G \cup V_{G'}$.\footnote{It is possible that either $\vv y_0 \not\in V_G$ or $\vv y_0 \not\in V_{G'}$. Then one side of (\ref{eq:DE}) is an empty sum, which by convention is $\vv 0$.}\\

Note that in the associated dynamical system of a mass-action system,  $\frac{d\vv x}{dt}$ belongs to the stoichiometric subspace $S$.  Moreover, $\rrpp^n$ is forward invariant under mass-action kinetics, i.e., if $\vv x(0) \in \rrpp^n$, then $\vv x(t) \in \rrpp^n$ for all $t \geq 0$~\cite{FeinLectNts}. Consequently, the trajectory $\vv x(t)$ is confined to the stoichiometric compatibility class $(\vv x(0) + S)_>$ for all $t \geq 0$. \\

\noindent 
\rmk 
The stoichiometric subspaces for dynamically equivalent systems can in principle be different. However, the \textit{kinetic subspaces} for the two systems must be the same.\footnote{The \emph{kinetic subspace} of a dynamical system $\frac{d{\vv x}}{dt}={\vv f(\vv x)}$ on a domain $\Omega$ is the linear subspace generated by $\{{\vv f(\vv x) \colon \vv x} \in \Omega\}$~\cite{FH77}. For a mass-action system, the kinetic subspace is a subset of the stoichiometric subspace $S$.} For example, the system in Figure~\ref{fig:changestoic}(a), made of the reaction \ce{2X ->[$k$] X + Y}, is dynamically equivalent to the system in Figure~\ref{fig:changestoic}(b), consisting of the reactions \ce{2X ->[$k$] X + Y} and \ce{0 <-[$k'$] Y ->[$k'$] 2Y}. By definition, the two systems have different stoichiometric subspaces. However, in these systems, the trajectory starting at $\vv x_0 \in \rrpp^n$ is confined to the affine space $\vv x_0 + \rr (-1,1)^T$ because their kinetic subspace is $\rr (-1,1)^T$. 

	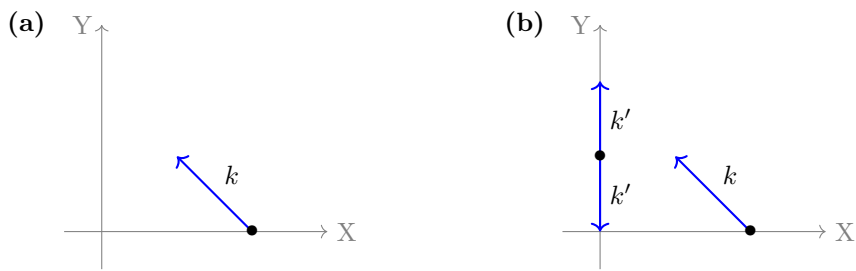
\begin{figure}[h!]
	\centering
		\begin{tikzpicture}
			\node at (-1,2.75) {\bf (a)};
			\draw [->, gray] (-0.5,0)--(3,0) node [right] {X};
			\draw [->, gray] (0,-0.5)--(0,2.75) node [left] {Y};
			\draw [->, blue, thick] (2,0)--(1,1) node [midway, above right] {\color{black}{$k$}};
			\node (2x) at (2,0) {$\bullet$};
		\end{tikzpicture}
		\hspace{1.5cm}
		\begin{tikzpicture}
			\node at (-1,2.75) {\bf (b)};
			\draw [->, gray] (-0.5,0)--(3,0) node [right] {X};
			\draw [->, gray] (0,-0.5)--(0,2.75) node [left] {Y};
			\draw [->, blue, thick] (2,0)--(1,1) node [midway, above right] {\color{black}{$k$}};
 			\draw [->, blue, thick] (0,1)--(0,0)node [midway, right] {\color{black}{$k'$}};
 			\draw [->, blue, thick] (0,1)--(0,2)node [midway, right] {\color{black}{$k'$}};
			\node (2x) at (2,0) {$\bullet$};
			\node (xy) at (0,1) {$\bullet$};
		\end{tikzpicture}
		\caption{Two dynamically equivalent systems with different stoichiometric subspaces. Trajectories are confined to the same affine invariant spaces because their kinetic subspaces are the same.}
		\label{fig:changestoic}
	\end{figure}

\begin{ex}
\label{ex:3systems}
	For the networks in Figure~\ref{fig:excrnrev}, let $\kk_{ij} >0$ be the rate constant on the reaction $\vv y_i \FR \vv y_j \in G$; let $\kk'_{ij}$ be the rate constant on the reaction $\vv y_i \FR \vv y_j \in {G'}$.  Suppose $\kk_{ij}$ and $\kk_{pq}'$ satisfy the following equations:
    \eq{
        \kk_{15} \begin{pmatrix} \,\,\, 1\, \, \\ 1   \end{pmatrix}
            &= \kk'_{12} \begin{pmatrix} \,\,\, 0 \,\, \\ 2 \end{pmatrix}
            + \kk'_{13} \begin{pmatrix} \,\,\, 3 \, \,  \\ 2 \end{pmatrix} 
            + \kk'_{14} \begin{pmatrix} \,\,\, 3 \, \,  \\ 0 \end{pmatrix},
        \quad 
        \\
        \kk_{25} \begin{pmatrix} 1 \\ -1 \end{pmatrix}
        &    = \kk'_{21} \begin{pmatrix} 0 \\ -2 \end{pmatrix}
            + \kk'_{23} \begin{pmatrix} \,\,\, 3\,\, \\ 0 \end{pmatrix}
            + \kk'_{24}  \begin{pmatrix} 3 \\ -2  \end{pmatrix}, \\
        \kk_{36}  \begin{pmatrix} -1 \\ -1   \end{pmatrix}
            &= \kk'_{31}  \begin{pmatrix} -3 \\ -2   \end{pmatrix}
            + \kk'_{32}  \begin{pmatrix} -3 \\ 0   \end{pmatrix}
            + \kk'_{34}  \begin{pmatrix} 0 \\ -2   \end{pmatrix},
        \quad 
        \\
        \kk_{46}  \begin{pmatrix} -1 \\ 1   \end{pmatrix}
        &    = \kk'_{41}  \begin{pmatrix} -3 \\ 0   \end{pmatrix}
            + \kk'_{42}  \begin{pmatrix} -3 \\ 2   \end{pmatrix}
            + \kk'_{43}  \begin{pmatrix}\,\,\, 0 \,\, \\ 2   \end{pmatrix}.
    }
    Then $G_{\vv \kk}$ and $G'_{\vv \kk'}$ are dynamically equivalent. The linear constraints on the rate constants arise from vector decomposition of the reaction vectors starting at the source vertices of $G$ and $G'$.

    In fact, if $\vv\kk$, $\vv \kk'$, and $\vv\kk^*$, where $\vv\kk^*$ is a vector of rate constants for $G^*$, satisfy some linear relations, the three mass-action systems $G_{\vv \kk}$, $G'_{\vv \kk'}$ and $G^*_{\vv\kk^*}$ are dynamically equivalent. \\
\end{ex}

Mass-action systems give rise to very diverse dynamics. For example, weakly reversible deficiency zero mass-action systems have exactly one locally asymptotically stable steady state (within the same stoichiometric compatibility class). Yet there are other mass-action systems that have periodic orbits or limit cycles~\cite{osc1, osc2,  osc3} and others that admit multiple steady states (within the same stoichiometric compatibility class)~\cite{cf05, cf06, Banaji_Craciun_2009}, and even chaotic dynamics~\cite{RevCY,chaosphys}. We refer the reader to \cite{RevCY, FeinLectNts, GunaNts,angeli_tutorial} for an introduction to mass-action systems.  In this paper, we focus on several kinds of steady states of mass-action systems.
  
\begin{defn}
\label{def:MASss}
	Let $G_{\vv\kk}$ be a mass-action system with the associated dynamical system
	\eq{
		\frac{d\vv x}{dt} \,\,=\!\! \sum_{\vv y\FR \vv y' \in G} \kk_{\vv y\FR \vv y'} \vv x^{\vv y} (\vv y' - \vv y).
	}
A state $\vv x_0\in \rrpp^n$ is a {\df{positive steady state}} if 
	\eqn{
		\label{def:eqMASss}
		\frac{d\vv x}{dt} 
		\,\,=  \!\!\sum_{\vv y\FR \vv y' \in G} \kk_{\vv y\FR \vv y'} \vv x_0^{\vv y} (\vv y' - \vv y) = \vv 0.
	}
	
\noindent
A positive steady state $\vv x_0 \in \rrpp^n$ is {\df{detailed-balanced}} if for every $\vv y \RR \vv y' \in G$, we have
	\eqn{
		\label{def:eqMASDB}
		 \kk_{\vv y\FR \vv y'} \vv x_0^{\vv y} = \kk_{\vv y' \FR \vv y} \vv x_0^{\vv y'}.
	}
	
\noindent
A positive steady state $\vv x_0 \in \rrpp^n$ is {\df{complex-balanced}} if for every vertex $\vv y_0 \in V_G$, we have
	\eqn{
		\label{def:eqMASCB}
		 \sum_{\vv y_0 \FR \vv y' \in G}\kk_{\vv y_0\FR \vv y'} \vv x_0^{\vv y_0} 
		 \,\,
		 =\!\!
		 \sum_{\vv y\FR \vv y_0 \in G }\kk_{\vv y\FR \vv y_0} \vv x_0^{\vv y}.
	}
\end{defn}

Intuitively, detailed balancing is when fluxes across every pair of reversible reactions are balanced; this is intimately related to the notion of microreversibility or dynamical equilibrium in physical chemistry~\cite{Boltzmann_1887, Boltzmann_1896}. Complex balancing is when fluxes through every vertex (i.e., reaction complex) is balanced.

\section{Fluxes on Reaction Networks}
\label{sec:FluxesCRN}

Most dynamical systems associated to reaction networks are nonlinear~\cite{HJ72,genMAS18, Savageau_Voit}. While nonlinear dynamical systems are generally difficult to study, the analysis of reaction networks is sometimes facilitated by the linear constraints arising from the network structure and stoichiometry. 

To illustrate what we mean, consider mass-action kinetics. The (generally nonlinear) dynamical system under mass-action kinetics has the form
	\eq{
		\frac{d\vv x}{dt} 
		\,\,=\!\!
		\sum_{\vv y \FR \vv y' \in G} 
			\nu_{\vv y \FR \vv y'}(\vv x)
			 (\vv y' - \vv y),
	}
where $\nu_{\vv y \FR \vv y'}(\vv x) = \kk_{\vv y\FR \vv y'} \vv x^{\vv y}$. Once the nonlinearity is hidden inside the reaction rate function $\nu_{\vv y \FR \vv y'}(\vv x)$, the linear structure remaining becomes apparent. 

Enumerate the set of reactions, $E = \{ \vv y_j \FR \vv y_j '\}_{j=1}^{|E|}$, and let $\vv \nu(\vv x) = (\nu_{\vv y_j \FR \vv y_j'}(\vv x))_{j=1}^{|E|}$ be a vector consisting of the reaction rate functions. Define the \df{stoichiometric matrix} $N \in \rr^{n \times |E|}$ as the matrix whose $j$th column is the $j$th reaction vector $\vv y'_j - \vv y_j$. Then the dynamical system above can be written succinctly as $\frac{d\vv x}{dt} = N\vv \nu(\vv x)$. 

In order to deal with the underlying linear structure, we do not keep track of the concentrations that give rise to $\vv \nu(\vv x)$ but leave it as a vector of unknowns. For this reason, we denote the value $\vv \nu(\vv x)$ simply as $\vv J$ and call it a \emph{flux vector}. 

\begin{defn}
\label{def:fluxsyst}
    A {\df{flux vector}} $\vv J =(J_{\vv y\FR \vv y'})_{\vv y\FR \vv y' \in G} \in \rrpp^{E}$ on a reaction network $G = (V,E)$ is a vector of positive numbers. The number $J_{\vv y\FR \vv y'}$ is called the \df{flux} of the reaction $\vv y\FR \vv y'$, and the pair $(G,\vv J)$ is called a \df{flux system}.\footnote{   
        The word ``system'' in ``flux system" is in the sense of a system of linear equations, rather than a dynamical system.
    }
\end{defn}

As with the rate constants, it may be convenient to refer to $J_{\vv y \to \vv y'}$ even when $\vv y \to \vv y' \not\in G$, in which case $J_{\vv y \to \vv y'} = 0$.\\
    
This idea of fluxes on a reaction network may be familiar to anyone who has worked with stoichiometric network analysis or flux balance analysis. One form of the analysis is to solve the linear equation $N \vv J = \vv 0$, where the unknown vector $\vv J$ has nonnegative coordinates~\cite{FBARev, FBARev2}. Since we are interested in relating network structure with dynamics, if $\vv y \to \vv y'\in G$, we impose that $J_{\vv y \to \vv y'} > 0$. Also if $\vv y \RR \vv y'$ is a reversible reaction in $G$, then $J_{\vv y \to \vv y'}$ and $J_{\vv y' \to \vv y}$ are two positive components of the vector $\vv J$. A solution $\vv J > \vv 0$ of the equation $N\vv J = \vv 0$ corresponds to a positive steady state if $\vv J = \vv \nu(\vv x_0)$ for some $\vv x_0 \in \rrpp^n$. We define the flux analogues of positive steady state, detailed-balanced steady state, and complex-balanced steady state.

\begin{defn}
\label{def:ssflux}
A {\df{steady state flux}} on a network $G = (V,E)$ is a flux vector $\vv J \in \rrpp^E$ satisfying
    \eqn{
    \label{def:SSF}
        \sum_{\vv y \FR \vv y' \in G} J_{\vv y\FR \vv y'} (\vv y' - \vv y) = \vv 0.
    }
A flux $\vv J\in \rrpp^E$ is said to be {\df{detailed-balanced}} if for every $\vv y \FR \vv y' \in G$, we have 
	\eqn{
	\label{def:DBF}
		 J_{\vv y\FR \vv y' } = J_{\vv y' \FR \vv y}.
	}
A flux $\vv J \in \rrpp^E$ is said to be {\df{complex-balanced}} if for every $\vv y_0 \in V$, we have
    \eqn{
    \label{def:CBF}
        \sum_{\vv y_0 \FR \vv y' \in G} J_{\vv y_0 \FR \vv y'} = \sum_{\vv y \FR \vv y_0 \in G} J_{\vv y \FR \vv y_0}.
    }
\end{defn}

A steady state flux is a positive vector $\vv J$ in $\ker N$, where the stoichiometric matrix $N$ has the reaction vectors as its columns. As a shorthand, we refer to the flux system $(G,\vv J)$ as detailed-balanced if $\vv J$ is a detailed-balanced flux on $G$. Similarly defined is a complex-balanced flux system on $G$. It will be clear from context whether a complex-balanced system refers to a mass-action system or a flux system. 

\begin{ex}
\label{ex:fluxsyst}
    An example of a flux system $(G, \vv J)$ is shown in Figure~\ref{fig:fluxsystem}. The positive number labelled on each edge $\vv y \to \vv y'$  is the flux $J_{\vv y \to \vv y'}$ of that reaction. 

    Note that this flux system could have risen from a mass-action system. For example, suppose the numbers labelled on the edges are taken to be rate constants $\vv k$, and the state of the system is $\vv x = \vv 1$. Then $(G, \vv J)$ would be the flux system based off of the mass-action system $G_{\vv\kk}$.

    There is \emph{no unique} mass-action system that gives rise to a fixed flux system. For example, on the reaction network shown in Figure~\ref{fig:fluxsystem}, suppose that the rate constants are taken to be 
	\eq{\begin{array}{rrr}
		\kk'_{\footnotesize \cf{0}\to \cf{Y}} = 3,
		& \qquad 
		\kk'_{\footnotesize \cf{Y}\to\cf{X}+\cf{Y}} =1,
		& \qquad 
		\kk'_{\footnotesize \cf{X}+\cf{Y}\to \cf{0}} = 1 ,
		\ds\vphantom{\frac{}{}}
		\\
		\kk'_{\footnotesize \cf{Y}\to \cf{0}} = \frac{1}{2},
		&
		\kk'_{\footnotesize \cf{X}+\cf{Y}\to 2\cf{X}} = \frac{5}{2},
		&
		\kk'_{\footnotesize 2\cf{X}\to \cf{X}+\cf{Y}} = 5,
		\ds\vphantom{\frac{}{}}
		\end{array}
	}
    and that the state of the system is $\vv x_0 = (1,2)^T$; then it can be shown that $(G,\vv J)$ is the flux system of the mass-action system $G_{\vv\kk'}$ at the state $\vv x_0$.

    This flux system $(G, \vv J)$ is complex-balanced. For example, at the vertex $(0,1)$ corresponding to $\cf{Y}$, there is one reaction going into it with flux value $3$, and there are two reactions leaving this vertex, with sum of fluxes being $2 + 1 = 3$. 
\end{ex}
\medskip 

	\begin{figure}[h!]
	\centering
		\begin{tikzpicture}[scale=2]
			\draw [->, gray] (-0.25,0)--(2.5,0) node [right] {X};
			\draw [->, gray] (0,-0.25)--(0,1.5) node [left] {Y};
			\node (0) at (0,0) {$\bullet$};
			\node (2x) at (2,0) {$\bullet$};
			\node (xy) at (1,1) {$\bullet$};
 			\node (y) at (0,1) {$\bullet$};
 			\draw [-{>[harpoon]}, blue, thick, transform canvas={xshift=-1.5pt, yshift=0pt}] (0)--(y) node [midway, left] {\footnotesize{\textcolor{black}{3}}};
 			\draw [-{>[harpoon]}, blue, thick, transform canvas={xshift=1.5pt, yshift=0pt}] (y)--(0) node [midway, right] {\footnotesize{\textcolor{black}{1}}};
			\draw [->, blue, thick] (y)--(xy) node [midway, above] {\footnotesize{\textcolor{black}{2}}};
 			\draw [->, blue, thick] (xy)--(0)node [midway, below right] {\footnotesize{\textcolor{black}{2}}};
 			\draw [-{>[harpoon]}, blue, thick, transform canvas={xshift=1pt, yshift=1pt}] (xy)--(2x) node [midway, above right] {\footnotesize{\textcolor{black}{5}}};
 			\draw [-{>[harpoon]}, blue, thick, transform canvas={xshift=-1pt, yshift=-1pt}] (2x)--(xy) node [midway, below left] {\footnotesize{\textcolor{black}{5}}};
		\end{tikzpicture}
		\caption{An example of a flux system. The positive numbers on any edge $\vv y \to \vv y'$ is the flux $J_{\vv y \to \vv y'}$ of that reaction. Note that this flux system is complex-balanced.}
		\label{fig:fluxsystem}
	\end{figure}
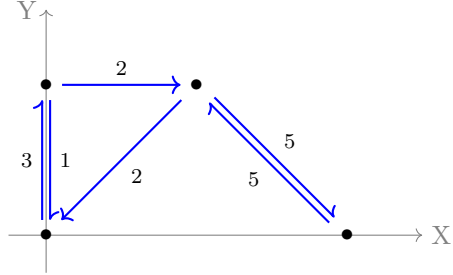

Whenever a flux vector arises from mass-action kinetics, i.e., $J_{\vv y\FR \vv y'} = \kk_{\vv y\FR \vv y'} \vv x^{\vv y}$, classical results for mass-action systems carry over to flux systems, as summarized in the following two lemmas.  

\begin{lem}
\label{lem:FluxSS}
	Let $G_\vv\kk$ be a mass-action system, and fix $\vv x\in \rrpp^n$. For each edge $\vv y \FR \vv y' \in G$, define $J_{\vv y \FR \vv y'} = \kk_{\vv y\FR \vv y'} \vv x^{\vv y}$, so that $\vv J = (J_{\vv y\FR \vv y'})_{\vv y\FR \vv y' \in G}$ is a flux vector on the network $G$. The following hold:
	\begin{enumerate}
	\item
		The flux vector $\vv J$ is a steady state flux on $G$ if and only if $\vv x$ is a positive steady state of $G_\vv\kk$. 
	\item
		The flux vector $\vv J$ is detailed-balanced if and only if $\vv x$ is a detailed-balanced steady state for $G_\vv\kk$.
	\item
		The flux vector $\vv J$ is complex-balanced if and only if $\vv x$ is a complex-balanced steady state for $G_\vv\kk$. 
	\end{enumerate}
\end{lem}

\begin{lem}
\label{lem:CBFWR}
If $G$ admits a detailed-balanced flux, then $G$ is reversible; if $G$ admits a complex-balanced flux, then $G$ is weakly reversible. If a flux is detailed-balanced on $G$, then it is also complex-balanced; if a flux is complex-balanced, then it is also a steady state flux.
\end{lem}
    \begin{proof}
        Let $ \vv J$ be a flux vector on a network $G$ --- either detailed-balanced or complex-balanced or merely a steady state flux. On $G$, define a mass-action system $G_\vv\kk$ with rate constants $k_{\vv y \to \vv y'} = J_{\vv y \to \vv y'}$ for each $\vv y \to \vv y' \in G$. Then $\vv x_0 = (1,\cdots, 1)^T$ is a (detailed-balanced or complex-balanced or positive) steady state. 
        
        Lemma~\ref{lem:CBFWR} follows from classical results on mass-action systems~\cite{GunaNts, FeinLectNts, Feinberg_1987, Fein72, Horn72, Horn_1974}. 
    \end{proof}

${}$

As we have seen in the previous section, some mass-action systems are dynamically equivalent; similarly there are flux equivalent systems. We define an equivalence relation for flux systems in  $\rr^n$.

\begin{defn}
\label{def:FE}
    Two flux systems $(G,\vv J)$ and $(G', \vv J')$ are {\df{flux equivalent}} if for every vertex $\vv y_0 \in V_G \cup V_{G'}$,\footnote{As before, we adopt the convention that the empty sum is $\vv 0$, i.e., $\sum_{\vv y \FR\vv  y' \in \emptyset} J_{\vv y\FR \vv y'} (\vv y' - \vv y) = \vv 0$.} 
    we have
    \eqn{\label{eq:FE}
        \sum_{\vv y_0 \FR \vv y \in G} J_{\vv y_0 \FR \vv y} (\vv y - \vv y_0) 
        \, \,  =\!\! \sum_{\vv y_0 \FR \vv  y' \in {G'}} J'_{\vv y_0 \FR \vv y'} (\vv y' - \vv y_0).
    }
    We denote equivalent flux systems by $(G, \vv J) \sim (G', \vv J')$ and say that $(G', \vv J')$ is a realization of $(G, \vv J)$. 
\end{defn}

\begin{lem}
Flux equivalence is an equivalence relation.
\end{lem}
	\begin{proof}
		That flux equivalence is symmetric and reflexive is clear. Suppose $(G, \vv J) \sim (G', \vv J')$ and $(G',\vv J') \sim (G^*, \vv J^*)$. 
		Transitivity follows from
		\eq{
			\sum_{\vv y_0 \FR \vv y \in G} J_{\vv y_0 \FR \vv y} (\vv y - \vv y_0)
			\,\,=\!\!
			\sum_{\vv y_0 \FR \vv y \in G'} J'_{\vv y_0 \FR \vv y} (\vv y - \vv y_0) 
			\,\,=\!\!
			\sum_{\vv y_0 \FR \vv y \in G^*} J^*_{\vv y_0 \FR \vv y} (\vv y - \vv y_0)
		}
for any $\vv y_0 \in V_{G} \cup V_{G'} \cup V_{G^*}$. Note that if $\vv y_0 \not\in V_{G'}$, then the sums above are all $\vv 0$. 
	\end{proof}	

Suppose a flux vector arises from a mass-action system; one expects the notion of dynamical equivalence  to line up with that of flux equivalence. 

\begin{prop}
\label{prop:DEMAS}
Let $G_\vv \kk$, $G'_{\vv \kk'}$ be mass-action systems, and fix $\vv x \in \rrpp^n$. For each edge $\vv y \to \vv y' \in G$, let $J_{\vv y \FR \vv y'} = \kk_{\vv y \FR \vv y'} \vv x^{\vv y} $, so that $\vv J(\vv x) = (J_{\vv y\FR \vv y'})_{\vv y \FR \vv y' \in G} $ is a flux vector on $G$. Similarly, define the flux vector $\vv J'(\vv x) = (J'_{\vv y \FR\vv  y'})_{\vv y \FR \vv y' \in G'} $ on $G'$, where $J'_{\vv y \FR \vv y'} = \kk'_{\vv y \FR \vv y'} \vv x^{\vv y} $. Then the following are equivalent:
	\begin{enumerate}
	\item
		The mass-action systems $G_\vv \kk$ and $G'_{\vv \kk'}$ are dynamically equivalent.
	\item
		The flux systems $(G, \vv J(\vv x))$, $(G', \vv J'(\vv x))$ are flux equivalent for all $\vv x \in \rrpp^n$.
	\item
		The flux systems $(G, \vv J(\vv x))$, $(G', \vv J'(\vv x))$ are flux equivalent for some $\vv x \in \rrpp^n$.
	\end{enumerate}
\end{prop}
	\begin{proof}
		It is clear that statements 1 and 2 are equivalent, and that statement 2 implies statement 3. Showing the implication of statement 1 from statement 3 will complete the proof. Let $\vv x_0 \in \rrpp^n$ be a vector such that $(G, \vv J(\vv x_0)) \sim (G', \vv J'(\vv x_0))$. For any $\vv y_0 \in V_G \cup V_{G'}$ and arbitrary $\vv x \in \rrpp^n$, we have
		\eq{
			&\hphantom{11}\,
			\sum_{\vv y_0 \FR \vv y \in G} J_{\vv y_0 \FR \vv y}(\vv x) (\vv y - \vv y_0)
			\,\,-\!\!
			\sum_{\vv y_0 \FR \vv y' \in G'} J'_{\vv y_0 \FR \vv y'}(\vv x) (\vv y' - \vv y_0)
			\\&= 
			\sum_{\vv y_0 \FR \vv y \in G} \kk_{\vv y_0 \FR \vv y} \vv x^{\vv y_0} (\vv y - \vv y_0)
			\,\,-\!\!
			\sum_{\vv y_0 \FR \vv y' \in G'} \kk'_{\vv y_0 \FR \vv y'} \vv x^{\vv y_0} (\vv y' - \vv y_0)
			\\&= 
			\frac{\vv x^{\vv y_0}}{\vv x_0^{\vv y_0}} \left(
			\sum_{\vv y_0 \FR \vv y \in G} \kk_{\vv y_0 \FR \vv y} \vv x_0^{\vv y_0} (\vv y - \vv y_0)
			\,\,-\!\!
			\sum_{\vv y_0 \FR \vv y' \in G'} \kk'_{\vv y_0 \FR \vv y'} \vv x_0^{\vv y_0} (\vv y' - \vv y_0)
			\right)
			\\&= 
			\frac{\vv x^{\vv y_0}}{\vv x_0^{\vv y_0}} \left(
			\sum_{\vv y_0 \FR \vv y \in G} J_{\vv y_0 \FR \vv y} (\vv x_0) (\vv y - \vv y_0)
			\,\,-\!\!
			\sum_{\vv y_0 \FR \vv y' \in G'} J'_{\vv y_0 \FR \vv y'} (\vv x_0) (\vv y' - \vv y_0)
			\right)
			\\&= \vv 0.
		}
	\end{proof}

\noindent
\rmk 	
The proof above holds for kinetics other than mass-action type. For each (source) vertex $\vv y \in V_{G} \cup V_{G'}$, define a rate function $\nu_{\vv y}\colon \rrpp^n \to \rrpp$. Then the above proposition holds when the flux vectors are defined to be $J_{\vv y \to \vv y'} = \kk_{\vv y \to \vv y'} \nu_{\vv y}(\vv x)$ for each $\vv y \to \vv y' \in G$, and $J'_{\vv y \to \vv y'} = \kk'_{\vv y \to \vv y'} \nu_{\vv y}(\vv x)$ for each $\vv y \to \vv y' \in G'$. \\

In the following proposition, we reduce a nonlinear problem about mass-action systems to a linear problem about flux systems. Instead of showing that a mass-action system is dynamically equivalent to a complex-balanced (or detailed-balanced) system, it suffices to show that an appropriately defined flux system is flux equivalent to a complex-balanced (or detailed-balanced) system. 

\begin{prop}
\label{prop:DEtool}
	Let $G_{\vv \kk}$ be a mass-action system, and let $\vv x_0 \in \rrpp^n$. For each edge $\vv y \FR \vv y' \in G$, define $J_{\vv y \FR \vv y'} = \kk_{\vv y\FR \vv y'} \vv x_0^{\vv y}$, so that $\vv J = (J_{\vv y\FR \vv y'})_{\vv y\FR \vv y' \in G}$ is a flux vector on the network $G$. Suppose $(G, \vv J)$ is flux equivalent to $(G', \vv J')$, where $\vv J'$ is complex-balanced; then $G_\vv\kk$ is dynamically equivalent to a mass-action system $G'_{\vv\kk'}$, where $\vv x_0$ is a complex-balanced steady state for $G'_{\vv\kk'}$. Similarly, if $(G,\vv J)$ is flux equivalent to a detailed-balanced flux system $(G', \vv J')$, then $G_\vv\kk$ is dynamically equivalent to a mass-action system $G'_{\vv\kk'}$, where $\vv x_0$ is a detailed-balanced steady state for $G'_{\vv\kk'}$.
\end{prop}
	\begin{proof}
		For each edge $\vv y \to \vv y' \in G'$, define its rate constant to be 
		    \eq{
		        \kk'_{\vv y \to \vv y'} = \frac{J'_{\vv y \to \vv y'}}{\vv x_0^{\vv y}} > 0,
		    }
 		so that $G'_{\vv \kk'}$ is a mass-action system. By Proposition~\ref{prop:DEMAS}, the mass-action systems $G_\vv\kk$ and $G'_{\vv \kk'}$ are dynamically equivalent, and by Lemma~\ref{lem:FluxSS}, $\vv x_0$ is a complex-balanced steady state if $\vv J'$ is a complex-balanced flux on $G'$, and if $\vv J'$ is detailed-balanced on $G'$, then $\vv x_0$ is a detailed-balanced steady state.
	\end{proof}

\section{Complex balancing without additional vertices}
\label{sec:NoNewNodes}

The identification of possible network structures associated to a biochemical system, say, from experimental data, is closely related to identifying key players in the system (e.g., enzymes in metabolic networks, genes in genetic networks). While the general nonuniqueness implies that network identification may often be impossible, it may still be desirable to compute equivalent systems --- whether that be dynamical equivalence or flux equivalence --- in order to conclude that the system has better properties than first suspected, e.g., weak reversibility or  complex balance. This problem is not new~\cite{Sze2011, cpIdentifiability}. 

In recent years, the engineering community has utilized properties of mass-action systems in novel ways to designing and analyzing control systems~\cite{angeli_tutorial, Sontag1, KineticFeedbackDesign,DEMASControl}. For example, the controllers can be added in such a way that the resulting system is a complex-balanced mass-action system; from this, one can conclude that the control system has a unique positive steady state and  local stability~\cite{KineticFeedbackDesign,DEMASControl}. Moreover, very general results have been obtained on the stability of complex-balanced systems with delay~\cite{delayCB}.
	
Thus, there is strong incentive for developing effective computational methods to find structurally better dynamically equivalent systems. One approach uses linear programming, but an objective function must be chosen. To reduce the search space, one can decide to search for a realization with  the maximal and minimal number of edges~\cite{WRAlgDense, WRAlgSparse}. Nonetheless, the set of vertices to be included in the reaction network must be chosen ahead of time.  
	
In the examples of Figure~\ref{fig:changestoic}, the mass-action systems systems are dynamically equivalent, but one uses an additional source vertex, whose weighted vectors sum to zero. Intuition may say that additional vertices can only improve the chance to find a network with desirable properties, as additional parameters provide extra degrees of freedom. Even if that is the case, the question of {\it computability} arises. Even if by adding new vertices to the network, one can produce an equivalent complex-balanced system, there is no a priori bound on the number of new vertices needed. One cannot realistically add new vertices ad infinitum. 
	
Fortunately, we prove that {\em no additional vertices are needed} in order to check if a given system admits complex-balanced realizations. Thus, to check whether or not a network can admit a complex-balanced realization becomes a finite calculation, one that can be done by searching through the admissible domain as done in linear programming. Although the motivation came from mass-action systems, we prove our results in the more general setting of flux systems.
	
Our approach is to show that any such additional vertices in the network can be removed without changing the properties desired, namely, complex-balanced or weak reversibility. Such additional vertices will be called \emph{virtual sources}. 
	
\begin{defn}
    A vertex $\vv y_0 \in V_{s}$ is a \df{virtual source} of the flux system $(G,\vv J)$ if 
    \eqn{
        \sum_{\vv y_0 \to \vv y' \in G} J_{\vv y_0 \to \vv y'} (\vv y' - \vv y_0) = \vv 0,
    }
    where the sum is over all edges with $\vv y_0$ as its source.
\end{defn}

If the flux system $(G, \vv J)$ arises from a mass-action system, then $\vv y_0 \in V_s$ is a virtual source if and only if the monomial $\vv x^{\vv y_0}$ does \emph{not} appear\footnote{That is the monomial does not appear after simplifying.} on the right-hand side of the associated dynamical system (\ref{eq:eqMAS}). For example, if we consider fluxes that arise from mass-action kinetics in the network in Figure~\ref{fig:changestoic}(b), the vertex $\cf{Y}$ is a virtual source.

\bigskip	
	
In this section, we prove that if a flux vector on a weakly reversible reaction network is complex-balanced and has a virtual source $\vv y^*$, then there is an equivalent complex-balanced flux system that does not involve $\vv y^*$ at all. In short, virtual sources $\vv y^*$ are not needed for complex balancing. 
	
Just as an arbitrary concentration vector $\vv x \in \rrpp^n$ may not be a complex-balanced steady state for a weakly reversible mass-action system, so we may want to speak of fluxes that are not complex-balanced. To keep track of how far a flux vector is from being complex-balanced, we define the \emph{potential} at a vertex to be the difference between incoming and outgoing fluxes.

\begin{defn}
\label{def:potential}
	Let $G = (V,E)$ be a reaction network, and let $\vv J \in \rrpp^{E}$ be a flux vector on $G$. The \df{potential} at a vertex $\vv y^* \in V$ is the scalar quantity
	\eqn{
		P_{(G, \vv J)}(\vv y^*)
		\,\,=\!\!
		\sum_{\vv y \FR \vv y^* \in G} J_{\vv y \FR \vv y^*}
		\, \,-
		\sum_{\vv y^* \FR \vv y' \in G} J_{\vv y^* \FR \vv y'} .
	}
\end{defn}

\rmk 
The flux vector $\vv J$ is complex-balanced on $G$ if and only if $P_{(G,\vv J)}(\vv y) = 0$ for all $\vv y \in V_s$. \\

By an abuse of notation, if $\vv y^* \not\in G$, we still refer to the potential $P_{(G,\vv J)}(\vv y^*)$ by setting it to be $P_{(G,\vv J)}(\vv y^*) = 0$. 

In showing that virtual sources are not needed for complex balancing, the idea is to redirect the fluxes flowing into a virtual source $\vv y^*$ to other vertices while maintaining flux equivalence. If we are doing nothing more than redirecting flow of fluxes, the potential at every vertex does not change; therefore, we preserve complex balancing for the resulting flux system. This type of construction appeared first in~\cite{KineticFeedbackDesign} to show that new monomials were not necessary in feedback design.

We have to simultaneously keep track of the potential at each vertex and flux equivalence. We illustrate the key idea of Lemma~\ref{lem:nDbasecase} in Figure~\ref{fig:nDbasecase}. 

	\begin{figure}[h!]
		\centering
		\begin{tikzpicture}
			\node (t) at (2,1.3) {\blue{$\bullet$}};
			\node (1) at  (0, 0) {$\bullet$};
			\node (2) at (2.75,-0.25) {$\bullet$};
			\node (3) at (3.5,1.2) {$\bullet$};
			\node (4) at (1.75,3) {$\bullet$};
			\node (0) at (5, 2) {$\bullet$};
			\node [left = 0pt of t] {\blue{$\vv y^*$}};
			\node [left = 0pt of 1] {$\vv y_1$};
			\node [below right = -7pt and -6pt of 3] {$\vv y_3$};
			\node [below = 0pt of 2] {$\vv y_2$};
			\node [above = 0pt of 4] {$\vv y_4$};
			\node [right =0pt of 0] {$\vv z$};
			\draw [->] (t)--(1);
			\draw [->] (t)--(2);
			\draw [->] (t)--(3);
			\draw [->] (t)--(4);
			\draw [->] (0)--(t);
			\draw [dashed, opacity=0.2] (1)--(2)--(3)--(4)--(1)--(3)--(2)--(4);
		\end{tikzpicture}
		 \hspace{1cm}
		\begin{tikzpicture}
			\node (t) at (2,1.3) {\color{blue!35!white}{$\bullet$}};
			\node (1) at  (0, 0) {$\bullet$};
			\node (2) at (2.75,-0.25) {$\bullet$};
			\node (3) at (3.5,1.2) {$\bullet$};
			\node (4) at (1.75,3) {$\bullet$};
			\node (0) at (5, 2) {$\bullet$};
			\node [left = 0pt of t] {\color{blue!35!white}{$\vv y^*$}};
			\node [left = 0pt of 1] {$\vv y_1$};
			\node [below right = -7pt and -6pt of 3] {$\vv y_3$};
			\node [below = 0pt of 2] {$\vv y_2$};
			\node [above = 0pt of 4] {$\vv y_4$};
			\node [right =0pt of 0] {$\vv z$};
			\draw [->] (0)--(1);
			\draw [->] (0)--(2);
			\draw [->] (0)--(3);
			\draw [->] (0)--(4);
			\draw [dashed, opacity=0.2] (1)--(2)--(3)--(4)--(1)--(3)--(2)--(4);
		\end{tikzpicture}
		\begin{tikzpicture}[overlay]
			\node at (-14,4.5) {\bf (a) };
			\node at (-6.25,4.5) {\bf (b) };
		\end{tikzpicture}
		\caption{Illustrating the idea behind Lemma~\ref{lem:nDbasecase} in $\rr^3$. (a) Assume that $\vv y^*$ is a virtual source in the flux system $(G, \vv J)$. In (b) is an equivalent flux system $(G', \vv J')$, obtained by redirecting fluxes from $\vv z \to \vv y^* \to \vv y_j$ as fluxes from $\vv z \to \vv y_j$. }
		\label{fig:nDbasecase}
	\end{figure}
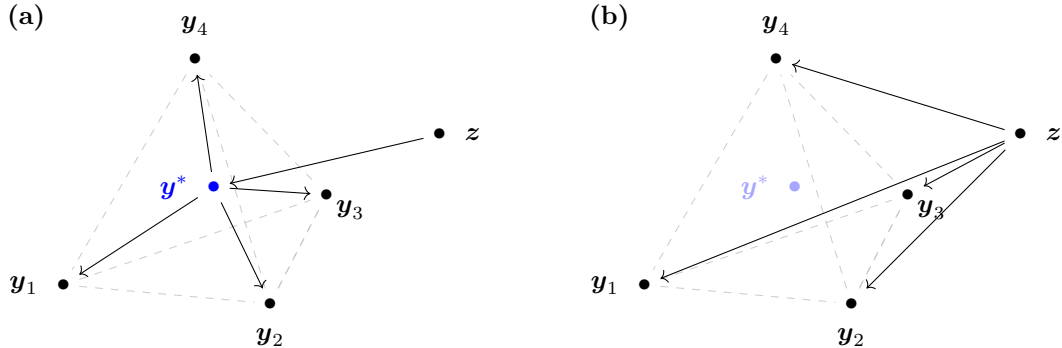

\begin{lem}
\label{lem:nDbasecase}
	Consider a reaction network $G$ consisting of the reactions $\vv z \to \vv y^*$ and $\vv y^* \FR \vv y_j$ for $j = 1,2,\dots, M$. Suppose $\vv y^*$ is a virtual source for a flux system $(G, \vv J)$ and its potential is $P_{(G,\vv J)}(\vv y^*) = 0$. Then there exists a flux equivalent system  $(G',\vv J')$ such that $\vv y^* \not\in V_{G'}$, and the potential at each vertex is preserved, i.e., $P_{(G,\vv J)}(\vv y_j) = P_{(G',\vv J')}(\vv y_j)$ for $1 \leq j \leq M$ and $P_{(G,\vv J)}(\vv z) = P_{(G',\vv J')}(\vv z)$. 

    The flux system $(G', \vv J')$ can be obtained constructively: remove the edges $\vv z \FR \vv y^*$ and $\vv y^* \FR \vv y_j$, and add the edges $\vv z \FR \vv y_j$ with fluxes $J'_{\vv z \FR \vv y_j} = J_{\vv y^* \FR \vv y_j}$. 
\end{lem}
	\begin{proof}
	For $j =1,2,\dots, M$, let $\vv w_j = \vv y_j - \vv y^*$ and $\vv w_0 = \vv z - \vv y^*$ denote the reaction vectors. First, remove the edges $\vv y^* \FR \vv y_j$ coming out of $\vv y^*$. Because $\vv y^*$ is a virtual source, $\sum_{j=1}^{M} J_{\vv y^* \FR \vv y_j} \vv w_j = \vv 0$, so the resulting flux system is still equivalent to the original. Note that in this new flux system, only $\vv z$ is a source vertex. 
	
	Next, we redirect the reaction $\vv z \FR \vv y^*$. Instead of the reaction $\vv z \FR \vv y^*$ with flux $J_{\vv z \FR \vv y^*}$, we have $M$ reactions $\vv z \FR \vv y_j$ with fluxes $J'_{\vv z \FR \vv y_j} = J_{\vv y^* \FR \vv y_j}$. Let $(G', \vv J')$ denote this newest flux system.
	
	Recall that flux equivalence means (\ref{eq:FE}) holds at each vertex of $G$ and $G'$. Here we only need to look at the vertex $\vv z$ to show that $(G', \vv J') \sim (G,\vv J)$. Note that $\vv y_j - \vv z = \vv w_j - \vv w_0$. From $P_{(G,\vv J)}(\vv y^*) = 0$, we also have $\sum_{j=1}^{M} J_{\vv y^* \FR \vv y_j} = J_{\vv z \FR \vv y^*}$. Thus, the weighted sum of vectors coming out of $\vv z$ is 
	\eq{
		\sum_{j=1}^{M} J'_{\vv z \FR \vv y_j} \left( \vv y_j - \vv z\right)
		&= \sum_{j=1}^{M} J_{\vv y^* \FR \vv y_j} \left( \vv w_j - \vv w_0 \right)
		= \underbrace{
				\sum_{j=1}^{M} J_{\vv y^* \FR \vv y_j} \vv w_j
			}_{=\,\, \vv 0}
			\, - \,  \,
			\vv w_0 \sum_{j=1}^{M} J_{\vv y^* \FR \vv y_j} 
		= - J_{\vv z \FR \vv y^*} \vv w_0,
	}
    and $(G', \vv J') \sim (G, \vv J)$. 
		
	Finally, we prove that the potentials are unchanged. Trivially $P_{(G,\vv J)}(\vv y^*) = P_{(G',\vv J')}(\vv y^*) = 0$. Also $P_{(G,\vv J)}(\vv y_j) =  J_{\vv y^*\FR \vv y_j} =  J'_{\vv z \FR \vv y_j} = P_{(G',\vv J')}(\vv y_j)$ for $ j = 1,2,\dots, M$. Last but not least,  
	\eq{
		- P_{(G',\vv J')}(\vv z) &= 
			\sum_{j=1}^{M} J'_{\vv z \FR \vv y_j}
		= \sum_{j=1}^{M} J_{\vv y^* \FR \vv y_j} 
		= J_{\vv z \FR \vv y^*}
		= - P_{(G,\vv J)}(\vv z).
	}
We have shown that the resulting flux system $(G',\vv J')$ is flux equivalent to the original flux system $(G, \vv J)$, and the potential at each vertex is preserved.
	\end{proof}

\rmk In Lemma~\ref{lem:nDbasecase}, the source vertex $\vv z$ may \emph{not} be distinct from $\vv y_j$. \\

We now arrive at our main technical theorem (Theorem~\ref{thm:NoNewNodesnD}), a generalization of Lemma~\ref{lem:nDbasecase}. Here, the virtual source $\vv y^*$ may have multiple reactions coming into it and coming out of it. The proof will be an induction on the number of edges flowing into $\vv y^*$. At each step, we redirect a fraction of the fluxes flowing through $\vv y^*$ from one incoming edge. 

\begin{thm}
\label{thm:NoNewNodesnD} 
    Let $(G, \vv J)$ be a complex-balanced flux system on reaction network $G = (V,E)$. Suppose that $\vv y^* \in V$ is a virtual source. Then there exists an equivalent complex-balanced flux system $(G', \vv J')$ with $V_{G'} = V \setminus \{ \vv y^*\}$. Moreover,
    \eqn{ 
    J'_{\vv y_i \FR \vv y_k} = J_{\vv y_i \FR \vv y_k} +  J_{\vv y^* \FR \vv y_k} \frac{ J_{\vv y_i \FR \vv y^*} }{\sum_{\vv y_j \to \vv y^* \in G} J_{\vv y_j \FR \vv y^*}}
    }
    for any $\vv y_i$ such that $\vv y_i \to \vv y^* \in G$ and any $\vv y_k$ such that $\vv y^* \to \vv y_k \in G$, and $J'_{\vv y \FR \vv y'} = J_{\vv y \FR \vv y'}$ for all other edges $\vv y \to \vv y' $. 
\end{thm}
	\begin{proof}
	Let $N$ be the number of reactions with $\vv y^*$ as target, i.e., $N = |\{ \vv z \FR \vv y^* \in G \}|$. Enumerate the sources as $\vv z_1,\vv z_2, \dots, \vv z_N$. Let $M$ be the number of reactions with $\vv y^*$ as sources, i.e., $M = |\{ \vv y^* \FR \vv y \in G \}|$. Enumerate the targets as $\vv y_1, \vv y_2, \dots, \vv y_M$. Since $\vv y^*$ is a virtual source, it is in the relative interior of the convex hull of the targets $\vv y_j$. From complex balancing, we have $P_{(G,\vv J)}(\vv y^*) = 0$, or 
	\eq{
		\sum_{j=1}^M J_{\vv y^* \FR \vv y_j} 
		= \sum_{i=1}^N J_{\vv z_i \FR \vv y^*}.
	} 

	Let $\theta =\frac{J_{\vv z_1 \FR \vv y^*}}{\sum_i  J_{\vv z_i \FR \vv y^*}}$ be the fraction of flux to be redirected from $\vv z_1 \FR \vv y^*$. We apply the construction described in Lemma~\ref{lem:nDbasecase} to the incoming edge $\vv z_1 \FR \vv y^*$, and the outgoing edges $\vv y^* \FR \vv y_j$ for $j =1,2,\dots, M$. Let $(G', \vv J')$ denote the flux system after the diversion. More precisely, $J'_{\vv z_1 \FR \vv y^*} = 0$,  
	    \eq{
	        J'_{\vv z_1 \FR \vv y_j} - J_{\vv z_1 \FR \vv y_j} &= \theta J_{\vv y^* \FR \vv y_j},
	        \\
	        J'_{\vv y^* \FR \vv y_j} - J_{\vv y^* \FR \vv y_j} &= -\theta J_{\vv y^* \FR \vv y_j},
	    }
	and the fluxes on all other edges unchanged from $\vv J$. 
	
	Checking for flux equivalence at $\vv z_1$ before and after the diversion, we see that
	    \eq{
	       &\left( \text{Final flow from $\vv z_1$} \right) - \left( \text{Initial flow from $\vv z_1$} \right)
	       \\&= \sum_{j=1}^M J_{\vv z_1 \to \vv y_j}(\vv y_j - \vv z_1)
	        - J_{\vv z_1 \FR \vv y^*} (\vv y^* - \vv z_1)
	       \\&= \theta 
	            \underbrace{
	                \sum_{j=1}^M J_{\vv y^* \FR \vv y_j} (\vv y_j - \vv y^*) 
	           }_{=\,\, \vv 0}
	            + \theta 
	                \underbrace{ 
	                    \sum_{j=1}^M  J_{\vv y^* \FR \vv y_j} 
	               }_{=\,\, \sum_{j=1}^N J_{\vv z_i \FR \vv y^*} }     
	               (\vv y^* - \vv z_1)  
	            - J_{\vv z_1 \FR \vv y^*} (\vv y^* - \vv z_1)
	       \\&= \vv 0.
	    }
    At all other vertices, the net flux is unchanged. 

	In terms of potentials, at $\vv z_1$, we have
	    \eq{
	        P_{(G',\vv J')}(\vv z_1) - P_{(G, \vv J)}(\vv z_1)
	        &= - \sum_{j=1}^M J'_{\vv z_1 \FR \vv y_j} + J_{\vv z_1 \FR \vv y^*}
	        = -  \theta \sum_{i=1}^N J_{\vv z_i \FR \vv y^*} + J_{\vv z_1 \FR \vv y^*} 
	        = 0.
	    }
	At each $\vv y_j$: 
	    \eq{
	        P_{(G',\vv J')}(\vv y_j) - P_{(G,\vv J)}(\vv y_j)
	        &= \left( J'_{\vv z_1 \FR \vv y_j} 
	                + J'_{\vv y^* \FR \vv y_j}
	            \right)
	            - \left(
	                J_{\vv z_1 \FR \vv y_j}
	                + J_{\vv y^* \FR \vv y_j} 
	            \right)
	        = 0.
	    }
	At $\vv y^*$: 
		\eq{
			P_{(G',\vv J')}(\vv y^*) - P_{(G,\vv J)}(\vv y^*)
	        = - J_{\vv z_1 \to \vv y^*} +\theta  \sum_{j=1}^M J_{\vv y^* \to \vv y_j} 
	        = 0.
		}
	
	The new flux system $(G', \vv J')$ after diverting the flux from $\vv z_1 \FR \vv y^*$ is still complex-balanced, as the potential is unchanged from those of $(G, \vv J)$. Moreover, $(G', \vv J')$ and $(G, \vv J)$ are flux equivalent. In addition, at $\vv y^*$, we have
	    \eq{
	        \sum_{\vv y^* \FR \vv y \in G'} J'_{\vv y^* \FR \vv y} (\vv y - \vv y^*) 
	       &= (1+\theta)\!\! \sum_{\vv y^* \FR \vv y \in G'}
	            J_{\vv y^* \FR \vv y} 
	        (\vv y - \vv y^*) 
	        = \vv 0,
	    }
	i.e, $\vv y^*$ is a virtual source for $(G', \vv J')$. 
	
	Thus we have recovered all the hypotheses stated in the theorem. The only difference between $(G, \vv J)$ and $(G', \vv J')$ is that $G'$ contains $N-1 = | \left\{ \vv z \FR \vv y^* \in G'\right\} |$ reactions with $\vv y^*$ as target vertex. By induction on the number $| \left\{ \vv z \FR \vv y^* \in G'\right\} |$, there exists a flux system $(G^*, \vv J^*)$ that is flux equivalent to $(G, \vv J)$, and for which $\vv J^*$ is a complex-balanced flux on $G^*$. Finally, because $P_{(G^*,\vv J^*)}(\vv y^*) = 0$, but there are no incoming reactions to $\vv y^*$, it follows that  there are no outgoing reactions from $\vv y^*$, i.e., $\vv y^* \not\in V_{G^*}$. 
	\end{proof}

When does a flux system (or a reaction network) admit a complex-balanced realization?  Theorem~\ref{thm:NoNewNodesnD} implies that virtual sources do not need to be considered. Theorem~\ref{thm:main1} below is the basis behind several relevant numerical methods in Section~\ref{sec:numerical} for determining if a flux system (or a reaction network) is equivalent to complex-balanced.

\begin{thm}
\label{thm:main1}
    Let $(G, \vv J)$ be a flux system, and $V_{G,s}$ its set of source vertices. Then $(G,\vv J)$ is flux equivalent to some complex-balanced flux system if and only if $(G,\vv J)$ is flux equivalent to some complex-balanced flux system $(G', \vv J')$ where $V_{G'} \subseteq V_{G,s}$.
\end{thm}
   \begin{proof}
	    One direction is trivial. To prove the other direction, suppose $(G,\vv J)$ is a flux system that is flux equivalent to some complex-balanced flux system $(\tilde G, \tilde{\vv J})$. If $\vv y^* \in V_{\tilde G} \setminus V_{G, s}$, the set $\{ \vv y^* \to \vv y \in G \}$ is empty; flux equivalent demands that
	        \eq{
	            \vv 0 = \sum_{\vv y^* \to \vv y \in {\tilde G} } \tilde J_{\vv y^* \to \vv y} (\vv y - \vv y^*).
	        }
	    Theorem~\ref{thm:NoNewNodesnD} implies we can maintain flux equivalence and complex balance even after dropping the vertex $\vv y^*$ from $V_{\tilde G}$. Repeating this process for all vertices not in $V_{G,s}$ ultimately implies that there is a complex-balanced flux system $(G', \vv J')$ such that $(G', \vv J') \sim (G, \vv J)$ and in addition $V_{G'} \subseteq V_{G,s}$.
	\end{proof}

\begin{thm}
\label{thm:main2}
    Let $G$ be a reaction network, and $V_{G,s}$ its set of source vertices. Then the following are equivalent:
    \begin{enumerate}[label=(\roman*)]
    \item 
        There exists a flux vector $\vv J$ such that $(G,\vv J)$ is flux equivalent to some complex-balanced flux system.
    \item
        There exists a flux vector $\vv J$ such that $(G,\vv J)$ is flux equivalent to some complex-balanced flux system $(G', \vv J')$, where $V_{G'} \subseteq V_{G,s}$.
    \end{enumerate}
\end{thm}
	\begin{proof}
	    The proof follows immediately from Theorem~\ref{thm:main1}.
	\end{proof}

\begin{thm}
    \label{thm:NoNewNodeCBMAS}
        A mass-action system $G_{\vv \kk}$ is dynamically equivalent to some complex-balanced  system if and only if it is dynamically equivalent to a complex-balanced  system $G'_{\vv \kk'}$ that only uses the source vertices, i.e., $V_{G'} \subseteq V_{G,s}$.
\end{thm}

	\begin{proof}
	    This theorem follows from Proposition~\ref{prop:DEMAS} and  Theorem~\ref{thm:main1}. Suppose $G_\vv\kk$ is dynamically equivalent to some complex-balanced mass-action system $\tilde{ G}_{\tilde{\vv\kk}}$. Define the appropriate fluxes $\vv J$ on $G$ and $\tilde{\vv J}$ on $\tilde{G}$; by Proposition~\ref{prop:DEMAS}, the two flux systems are flux equivalent. Theorem~\ref{thm:main1} holds if and only if $(G, \vv J)$ is flux equivalent to some complex-balanced flux system $(G', \vv J')$ where $V_{G'} \subseteq V_{G,s}$. Define the appropriate mass-action system $G'_{\vv\kk'}$ (see Proposition~\ref{prop:DEtool}); we have one direction of this theorem. The other direction is trivially true.  
	\end{proof}

All of our theorems thus far have been concerned with flux systems; in the case of mass-action systems, implicit in everything is the existence of a complex-balanced steady state. However, the idea of redirecting fluxes can be adapted to show the surprising result that weak reversibility can be accomplished (if at all) with no extra vertices.

\begin{thm}
\label{thm:NoNewNodeWR}
	A mass-action system $G_{\vv \kk}$ is dynamically equivalent to some weakly reversible mass-action system if and only if it is dynamically equivalent to a weakly reversible mass-action system $G'_{\vv \kk'}$ that only uses its source vertices, i.e., $V_{G'} \subseteq V_{G,s}$.
\end{thm}
	\begin{proof}
	    Without loss of generality, we may suppose that $G_{\vv \kk}$ is a weakly reversible mass-action system for which there exists a virtual source $\vv y^*$. 
        As in Theorem~\ref{thm:NoNewNodesnD}, we remove the vertex $\vv y^*$ by redirecting the reactions flowing through it. Since $G$ is weakly reversible, there exists some vertex $\vv z$ such that $\vv z \to \vv y^* \in  G$. As before, we will try to replace pairs of reactions $\vv z \to \vv y^*$ and $\vv y^* \to \vv y$ with $\vv z \to \vv y$. 
	    
	    Enumerate the set $\{ \vv y^* \to \vv y \in G\}$ as $\{ \vv y^* \to \vv y_i\}_{i=1}^M$, and enumerate the set $\{ \vv z \to \vv y^* \in G\}$ as $\{ \vv z_j \to \vv y^*\}_{j=1}^N$. For simplicity, let $\alpha_j = \kk_{\vv z_j \to \vv y^*}$, and let $\beta_i = \kk_{\vv y^* \to \vv y_i}$. Informally speaking, in place of the reactions $\vv z_j \to \vv y^*$ and $\vv y^* \to \vv y_i$, we shall have the reaction $\vv z_j \to \vv y_i$ with rate constant $\kk'_{\vv z_j \to \vv y_i} = \alpha_j  \frac{ \beta_i}{\sum \beta_s} $. More precisely, let $G'$ be the graph after deleting the vertex $\vv y^*$ and its adjacent edges from $G$, and (if needed) the edges $\vv z_j \to \vv y_i$ added for all $i=1,2,\dots, M$ and $j=1,2,\dots, N$. On $G'$, take the rate constants to be $\kk'_{\vv z_j \to \vv y^* }  = \kk'_{\vv y^* \to \vv y_i } = 0$ and 
	        \eq{
	            \kk'_{\vv z_j \to \vv y_i }= \kk_{\vv z_j \to \vv y_i} + \alpha_j \frac{ \beta_i}{\sum \beta_s} ,
	        }
	   and all other rate constants same as in $G_\vv \kk$.

	  The assumption that $\vv y^*$ is a virtual source can be written as
	    \eq{ 
	        \sum_{i = 1}^M \beta_i \vv y_i = \sum_{i=1}^M \beta_i \vv y^*.
	    }
    Now to check for dynamical equivalence at $\vv z_1$, we consider the differences due to the reactions $\vv z_1 \to \vv y_i$:  
	 \eq{ 
	    \sum_{i=1}^M (\kk'_{\vv z_1 \to \vv y_i} - \kk_{\vv z_1 \to \vv y_i}) (\vv y_i - \vv z_1) 
	    &=
	    \sum_{i=1}^M \alpha_1 \frac{ \beta_i}{\sum \beta_s} (\vv y_i - \vv z_1) 
	    \\&= \frac{\alpha_1}{\sum \beta_s} \left( \sum_{i=1}^M \beta_i \vv y^* - \sum_{i=1}^M \beta_i \vv z_1 \right)
	    \\&= \alpha_1 (\vv y^* - \vv z_1),
	 }
    which is the contribution from the reaction $\vv z_1 \to \vv y^*$. Since other reactions were untouched, we have dynamical equivalence at $\vv z_1$.	There is nothing special about $j=1$; the same holds for all source vertices $\vv z_2, \vv z_3,\dots, \vv z_N$. 
	
	Finally, given any cycle $\vv v_1 \to \vv v_2 \to \cdots \to \vv v_\ell \to \vv v_1$ in $G'$, whenever an edge $\vv z_j \to \vv y_i$ appears in the cycle, replace it with two edges $\vv z_j \to \vv y^* \to \vv y_i$, and obtain a cycle in $G$. Therefore, $G'$ is still weakly reversible.
	\end{proof}

We extend the above results (Theorems~\ref{thm:NoNewNodesnD}-\ref{thm:NoNewNodeWR}) to detailed-balanced fluxes and/or reversible networks. We summarize these results in the following theorems.
\begin{thm}
\label{thm:DBF}
    Let $(G, \vv J)$ be a detailed-balanced flux system on a reaction network $G = (V,E)$. Suppose that $\vv y^* \in V$ is a virtual source. 
Then there exists an equivalent detailed-balanced flux system $(G', \vv J')$ with $V_{G'} = V \setminus \{ \vv y^*\}$. Moreover,  
	\eqn{
        J'_{\vv y_i \to \vv y_k} = J_{\vv y_i \to \vv y_k} + J_{\vv y^* \to \vv y_k} \frac{J_{\vv y_i \to \vv y^*}}{\sum_{\vv y_j \to \vv y^* \in G} J_{\vv y_j \to \vv y^*} } 
    } 
for any $\vv y_i, \vv y_k$ connected to $\vv y^*$ in $(G,\vv J)$. Let other fluxes remain unchanged from $(G, \vv J)$. In particular, $(G, \vv J)$ is flux equivalent to some detailed-balanced flux system if and only if $(G, \vv J)$ is flux equivalent to some detailed-balanced flux system $(G', \vv J')$ where $V_{G'} \subseteq V_{G, s}$. 
\end{thm}
    \begin{proof}
        As in Theorem~\ref{thm:NoNewNodesnD}, we divert fluxes away from $\vv y^*$. We only need to check detail balancing. Consider any two vertices $\vv y_i \neq \vv y_k$ where $\vv y_i \RR \vv y^*, \vv y_k \RR \vv y^* \in G$. Using the fact that the flux system was originally detailed-balanced, i.e., $J_{\vv y \to \vv y'} = J_{\vv y' \to \vv y}$, we obtain
            \eq{
                J'_{\vv y_i \to \vv y_k}
                &= J_{\vv y_i \to \vv y_k} + J_{\vv y^* \to \vv y_k} \frac{J_{\vv y_i \to \vv y^*}}{\sum_{\vv y_j \to \vv y^* \in G} J_{\vv y_j \to \vv y^*} }
                \\&= J_{\vv y_k \to \vv y_i} + J_{\vv y^* \to \vv y_i} \frac{J_{\vv y_k \to \vv y^*}}{\sum_{\vv y_j \to \vv y^* \in G} J_{\vv y_j \to \vv y^*} }
                \\&= J'_{\vv y_k \to \vv y_i}.
            }
For any other pairs of reversible reaction, detail balancing is inherited from $(G, \vv J)$. In other words, $(G',\vv J')$ is detailed-balanced.
    \end{proof}

\begin{thm}
\label{thm:NoNewNodeRev}
    A mass-action system $G_{\vv\kk}$ is dynamically equivalent to some reversible system if and only if it is dynamically equivalent to a reversible system $G'_{\vv \kk'}$ that only uses its source vertices, i.e., $V_{G'} \subseteq V_{G,s}$. 
\end{thm}
    \begin{proof}
        We assume that $G_\vv\kk$ is reversible and has a virtual source $\vv y^* \in V_{G}$. We will replace the reactions $\{ \vv y^* \RR \vv y_i \in G\}$ by modifying/adding the reactions $\{ \vv y_i \RR \vv y_k \colon \vv y_i \RR \vv y^*, \vv y_k \RR \vv y^* \in G\}$.  For any $\vv y_i$, $\vv y_j$ such that $\vv y_i \RR \vv y^*$, $\vv y_k \RR \vv y^* \in G$, let $\kk'_{\vv y_i \to \vv y^*} = \kk'_{\vv y^* \to \vv y_i} = 0$ and 
        \eq{
        		\kk'_{\vv y_i \to \vv y_j}
        		= \kk_{\vv y_i \to \vv y_j} 
        			+ \kk_{\vv y_i \to \vv y^*} \left( \frac{\kk_{\vv y^* \to \vv y_j} }{\sum \kk_{\vv y^* \to \vv y_s}} \right).
        }
        Similar to Theorem~\ref{thm:NoNewNodeWR}, it can be shown that $G_{\vv\kk}$ and $G'_{\vv\kk'}$ are dynamically equivalent. Moreover, by symmetry of construction, $G'$ is reversible.
    \end{proof}

Note that related results have been obtained recently for the problem of kinetic feedback design involving complex-balanced and weakly reversible systems~\cite{KineticFeedbackDesign}.
Here, for the problem of dynamical equivalence, we show that a given system admits a dynamically equivalent system that is  complex-balanced (or weakly reversible, or detailed-balanced, or reversible) if and only if such a system exists  using {\em only} the complexes that are  already present in the original system.

\subsection{Connection to deficiency theory}

Within the reaction network theory literature, \emph{deficiency} is a well-known quantity defined for a network $G$. Equipped with mass-action kinetics, networks with low deficiency are known to enjoy special dynamical properties under mass-action kinetics. For example, the famous deficiency zero theorem says that a weakly reversible deficiency zero network is complex-balanced for any choices of rate constants~\cite{HJ72, Feinberg_1987}. As we have introduced, complex-balanced systems enjoy properties such as uniqueness and stability of steady states, existence of a Lyapunov function, and the steady states admit a monomial parametrization~\cite{GunaNts, FeinLectNts, RevCY, HJ72, Fein95}. Despite the strong implications, deficiency has a relatively simple definition.

\begin{defn}
Let $G = (V_G,E_G)$ be a reaction network with $\ell_G$ connected components. Suppose the dimension of the stoichiometric subspace $S$ is $s=\dim S$; then the {\df{deficiency}} of the network $G$ is the nonnegative integer 
	\eqn{
		\delta_G = |V_G| - \ell_G - s.
	}
\end{defn}

It can be shown that $\delta_G = \dim(\Ker Y \cap \Img I_G)$, where $Y$ is the stoichiometric matrix, with the vertices as its columns, and $I_G$ the incidence matrix of $G$~\cite{Joh14}. It follows that $\delta_G$ is a nonnegative integer. When the network is weakly reversible, we also have $\delta_G = \dim(\Ker Y \cap \Img A_{\vv\kk})$, where $- A_{\vv\kk}^T$ is the Laplacian of the weighted graph $G_{\vv\kk}$~\cite{GunaNts, FeinLectNts}.

Deficiency continues to play an important role in the analysis of reaction networks and mass-action systems. In our procedure for removing virtual vertices, deficiency always decreases. 
This is similar to a result obtained in \cite{KineticFeedbackDesign}, where the removal of additional monomials that function as controls in a feedback system also leads to a decrease in deficiency.

\begin{thm}
\label{thm:deficiency}
    Let $G_{\vv \kk}$ be a weakly reversible mass-action system with deficiency $\delta_G$. Suppose it has a virtual source $\vv y^*$. 
    Let $G'_{\vv\kk'}$ be the weakly reversible mass-action system as produced in Theorem~\ref{thm:NoNewNodeWR}, dynamically equivalent to $G_{\vv\kk}$ with $V_{G'} = V_{G} \setminus \{\vv y^*\}$. Then the deficiency of $G'_{\vv\kk'}$ is $\delta_{G'} = \delta_G - 1$.
\end{thm}
	\begin{proof}
	    In the proof of Theorem~\ref{thm:NoNewNodeWR}, we replaced the reactions $\vv z \to \vv y^*$ and $\vv y^* \to \vv y$ with the reaction $\vv z \to \vv y$ by choosing appropriate rate constants. It is clear that $|V_{G'}| = |V_{G}| - 1$, and the number of linkage classes stays the same. We claim that the stoichiometric subspace $S$ remains unchanged. Thus, the drop in deficiency is due to the removal of the vertex $\vv y^*$, and $\delta_{G'} = \delta_G - 1$.
	    
	    First enumerate the reactions coming out of $\vv y^*$ as $\vv y^* \to \vv y_j$, and enumerate the reactions going into $\vv y^*$ as $\vv z_i \to \vv y^*$. Let $S_0$ be the span of the reaction vectors ``untouched'' by our procedure, more precisely,
	        \eq{ 
	            S_0 = \Span_\rr \{ \vv y \to \vv y' \in G \colon \vv y \neq \vv y^* \text{ or } \vv y' \neq \vv y^*\}.
	        }
	    Let $S_G$ be the stoichiometric subspace of $G$, in particular,
	        \eq{
	            S_G = \Span_\rr \{ S_0, \, \vv y_j - \vv y^*, \, \vv y^* - \vv z_i\}_{i,j},
	        }
        and $S_{G'}$ be the stoichiometric subspace of $G'$, where
	        \eq{
	            S_{G'} = \Span_\rr \{ S_0,\,  \vv y_j - \vv z_i\}_{i,j}.
	        }
	    Clearly, $S_{G'} \subseteq S_G$, since $\vv y_j - \vv z_i = (\vv y_j - \vv y^*) + (\vv y^* -  \vv z_i) \in S_G$. Moreover, because $G$ is weakly reversible, the edge $\vv y^* \to \vv y_j$ is a part of a cycle; therefore, $S_G = \Span_\rr \{ S_0, \, \vv y^* - \vv z_i\}_{i}$. Finally, we note that $\vv y^*$ is in the convex hull of the vertices $\vv y_j$, and thus $\vv y^* - \vv z_i  \in \Span_\rr \{ \vv y_j - \vv z_i\}_j$, which implies $S_{G} \subseteq S_{G'}$. In other words, $S_G = S_{G'}$ and $\delta_{G'} = \delta_G - 1$.
	\end{proof}

\section{Numerical methods}
\label{sec:numerical}

In this section, we characterize when a flux system or a mass-action system is equivalent to a complex-balanced system. We also describe a method to determine when a mass-action system is dynamically equivalent to a complex-balanced or weakly reversible system. 

\subsection{Flux equivalence to complex-balancing}

Is a steady state flux system $(G,\vv J)$ flux equivalent to a complex-balanced one? The answer lies in the following linear feasibility problem for an unknown vector $\vv J'$. Enumerate the set of source vertices in $G$ as $\{\vv y_1, \vv y_2,\dots, \vv y_N\}$. Search for $\vv J' = (J'_{\vv y_i \to \vv y_j})_{i\neq j} \in \rr^{N ^2 - N}$ satisfying
    \eqn{
        \sum_{j \neq i} J'_{\vv y_i \to \vv y_j} (\vv y_j - \vv y_i) &= \sum_{\vv y_i \to \vv y \in G} J_{\vv y_i \to \vv y} (\vv y - \vv y_i)
            \qquad \text{ for } i = 1,2,\dots, N, \label{eq:Alg-FE} \\
        \sum_{j\neq i} J'_{\vv y_i \to \vv y_j} &= \sum_{j\neq i} J'_{\vv y_j \to \vv y_i}
            \qquad \text{ for } i =1,2,\dots, N, 
            \tag{18a}\label{eq:Alg-CBF} \\
        \addtocounter{equation}{1}
        \vv J' &\geq 0. \label{eq:Alg-Fpos} 
    }
If such a flux vector $\vv J'$ exists, then $(G, \vv J)$ is flux equivalent to a complex-balanced system. If no such flux vector $\vv J'$ exists, then $(G,\vv J)$ is not flux equivalent to a complex-balanced system. \\
    
Equation~(\ref{eq:Alg-FE}) is the flux equivalence condition, while (\ref{eq:Alg-CBF}) ensures that the new flux system is complex-balanced. Equation~(\ref{eq:Alg-FE}) alone checks for flux equivalence between any two given systems $(G,\vv J)$ and $(G', \vv J')$. \\

\begin{ex}
\label{ex:crnrevCtdFlux}
We return to the network $G$ in Figure~\ref{fig:excrnrev}(a) and Example~\ref{ex:3systems}. The network has $6$ vertices, $4$ of which are sources, and $4$ reactions. At the moment, we consider a flux system on the graph $G$ and ask, for what flux $\vv J$ is the flux system $(G,\vv J)$ equivalent to a complex-balanced one? One  can show that (\ref{eq:Alg-FE})-(\ref{eq:Alg-Fpos}) hold if and only if 
    \eqn{
    \label{eq:exFluxIneq}
        J_1=J_3,\quad  J_2=J_4,\quad \text{and } \quad   \frac{1}{5} < \frac{J_1}{J_2} < 5.
    }
A chosen flux $\vv J$ that satisfies (\ref{eq:exFluxIneq}) is flux equivalent to a complex-balanced system, whose network is a subgraph of $G'$ of Figure~\ref{fig:excrnrev}(b). 
The details of this characterization will be in an upcoming paper~\cite{STN}.
\end{ex}

\noindent
\rmk 
The setup for the detailed-balanced case is defined analogously. We keep (\ref{eq:Alg-FE}) and (\ref{eq:Alg-Fpos}) and include the equation
\eqn{
    J'_{\vv y_i \to \vv y_j} = J'_{\vv y_j \to \vv y_i}
    \qquad \text{ for } 1 \leq i \neq j \leq N. \tag{18b}
    \label{eq:Alg-DBF}
}

\subsection{Dynamical equivalence to complex balancing}

We considered above a set of equalities and inequalities necessary and sufficient for a flux system to be equivalent to a complex-balanced one. If the flux system arises from mass-action kinetics, we can write down an analogous system of equalities and inequalities necessary and sufficient for dynamical equivalence to a complex-balanced system. 
    
Consider a mass-action system $G_{\vv\kk}$, whose vertices are points in $\rr^n$, and enumerate the set of source vertices in $G$ as $\{ \vv y_1, \vv y_2,\dots, \vv y_N\}$. We set up a nonlinear feasibility problem for unknowns $\vv\kk'$ and $\vv x$. Search for vectors $\vv\kk' = (\kk'_{\vv y_i \to \vv y_j})_{i\neq j} \in \rr^{N^2 - N}$ and $\vv x \in \rr^n$ satisfying
    \eqn{
        \sum_{j\neq i} \kk'_{\vv y_i \to \vv y_j} (\vv y_j - \vv y_i)
        &= \sum_{\vv y_i \to \vv y \in G} \kk_{\vv y_i \to \vv y} (\vv y - \vv y_i)
        \qquad \text{ for } i = 1,2,\dots, N, \label{eq:Alg-DE} \\
        \sum_{j\neq i} \kk'_{\vv y_i \to \vv y_j} \vv x^{\vv y_i} 
        &= \sum_{j\neq i} \kk'_{\vv y_j \to \vv y_i} \vv x^{\vv y_j} 
         \qquad \text{ for } i =1,2,\dots, N, 
            \label{eq:Alg-CBss} \\
        \vv \kk' &\geq 0,
            \label{eq:Alg-kpos} 
        \\
        \vv x & > 0.
            \label{eq:Alg-sspos}
    }
If such $\vv \kk'$ and $\vv x$ exist, then $G_{\vv \kk}$ is dynamically equivalent to a complex-balanced system with $\vv x$ a complex-balanced steady state. If no such rate constants and steady state exist, then $G_{\vv \kk}$ is not dynamically equivalent to a complex-balanced system. \\
    
Equation~(\ref{eq:Alg-DE}) enforces dynamical equivalence. Equations~(\ref{eq:Alg-CBss}) and (\ref{eq:Alg-sspos}) imply that $\vv x$ is a positive complex-balanced steady state for an equivalent mass-action system; hence $\vv x$ is a positive steady state of $G_{\vv\kk}$. Note that in the inequality (\ref{eq:Alg-kpos}), some $\kk'_{\vv y_i\to \vv y_j}$ can be zero, which implies that $\vv y_i \to \vv y_j$ is not a reaction in the equivalent network. 

Equations~(\ref{eq:Alg-DE})-(\ref{eq:Alg-sspos}) generally form a nonlinear problem. Despite that, for networks with additional structure, one may be able to extract more information about the rate constants. One such example is the network $G$ in Figure~\ref{fig:excrnrev}(a). For this network we can completely characterize the parameter values for which  the associated mass-action system has a complex-balanced realization.

\begin{ex}
\label{ex:crnrevCtdMAS}
Consider a mass-action system on the network $G$ of Figure~\ref{fig:excrnrev}(a) and Example~\ref{ex:3systems}, with rate constants 
    \eq{
        \kk_{\vv y_1 \to \vv y_5} &= \kk_1, \quad 
        \kk_{\vv y_2 \to \vv y_5} = \kk_2, \quad
        \kk_{\vv y_3 \to \vv y_6} = \kk_3, \quad \text{and}\quad 
        \kk_{\vv y_4 \to \vv y_6} = \kk_4. 
    }
By a calculation, (\ref{eq:Alg-DE})-(\ref{eq:Alg-sspos}) hold if and only if 
    \eqn{
        \label{eq:exKIneq}
        \frac{1}{25} < \frac{k_1k_3}{k_2k_4} < 25.
    }
    
Again, a complex-balanced realization is a subgraph of $G'$ in Figure~\ref{fig:excrnrev}(b). More precisely, it is the reversible square with one pair of reversible diagonal (either $\vv y_1 \RR \vv y_3$ or $\vv y_2 \RR \vv y_4$); which diagonal is needed depends on the magnitudes of $k_1k_3$ and $k_2k_4$. The details of this characterization can be found in an upcoming paper~\cite{STN}.

The complex-balanced realization described (the subgraph of $G'$ in Figure~\ref{fig:excrnrev}(b)) has deficiency $\delta_{G'} = 1$. It is known that if its eight rate constants lie in a toric ideal of codimension $\delta_{G'} = 1$, then the mass-action system is complex-balanced~\cite{TDS}. While these eight rate constants are related to $\kk_1, \, \kk_2, \, \kk_3$, and $\kk_4$ by several linear equations, we found one explicit condition (\ref{eq:exKIneq}) for when the mass-action system $G_\vv\kk$ of Figure~\ref{fig:excrnrev}(a) is dynamically equivalent to a complex-balanced system. 

Finally, note that the network of  Example~\ref{ex:crnrevCtdMAS} gives rise to systems that are equivalent to complex-balanced for certain choices of rate constants, but {\em not} for other choices of rate constants. In a follow-up paper we will show that an entire class of networks give rise to systems that are equivalent to complex-balanced \emph{for all choice of rate constants}. More precisely, we will prove that systems generated by single-target networks that have their (unique) target vertex in the strict relative interior of the convex hull of its source vertices are dynamically equivalent to detailed-balanced mass-action systems for any choice of rate constants~\cite{STN}. 
\end{ex}

\subsection{Existence of a weakly reversible realization for a mass-action system}

While complex-balanced mass-action systems are weakly reversible, not all weakly reversible mass-action systems are complex-balanced. There has been much work on determining when a weakly reversible mass-action system is complex-balanced or not. Nonetheless, weakly reversible mass-action systems always have at least one positive steady state within each stoichiometric compatibility class~\cite{Boros_2018} and are conjectured to be persistent, and even permanent~\cite{CasianGAC1}.

We present a simple nonlinear feasibility problem to determine when a mass-action system is dynamically equivalent to a weakly reversible one. Recall that a mass-action system is weakly reversible if and only if it is complex-balanced for \emph{some} choice of rate constants. We introduce a scaling factor $\alpha_{\vv y_i \to \vv y_j}$ in order to decouple the dynamical equivalence condition from the complex-balanced condition. 

Consider a mass-action system $G_{\vv\kk}$, whose vertices are points in $\rr^n$, and enumerate the set of source vertices in $G$ as $\{\vv y_1, \vv y_2,\dots, \vv y_N\}$. We set up a nonlinear feasibility problem for unknown rate constants $\vv k'$ and a scaling factor $\vv\alpha$. Search for vectors $\vv \kk' = (\kk_{\vv y_i \to \vv y_j})_{i\neq j}$ and $\vv\alpha = (\alpha_{\vv y_i\to \vv y_j})_{i\neq j} \in \rr^{N^2-N}$ satisfying.
    \eqn{
        \sum_{j\neq i} \kk'_{\vv y_i \to \vv y_j} (\vv y_j - \vv y_i)
        &= \sum_{\vv y_i \to \vv y \in G} \kk_{\vv y_i \to \vv y} (\vv y - \vv y_i) 
        \qquad \text{ for } i = 1,2,\dots, N, \label{eq:Alg-DE2} \\
        \sum_{j\neq i} \alpha_{\vv y_i \to \vv y_j} \kk'_{\vv y_i \to \vv y_j} 
        &= \sum_{j\neq i} \alpha_{\vv y_j \to \vv y_i} \kk'_{\vv y_j \to \vv y_i}
         \qquad \text{ for } i =1,2,\dots, N, 
            \label{eq:Alg-CBss2} \\
        \vv \kk' &\geq 0,
            \label{eq:Alg-kpos2} 
        \\
        \vv \alpha &> 0.
            \label{eq:Alg-scalepos2} 
    }
If such $\vv\kk'$ and $\vv\alpha$ exist, then $G_{\vv\kk}$ is dynamically equivalent to a weakly reversible mass-action system. If no solution exists, then $G_\vv\kk$ is {\em not} dynamically equivalent to a weakly reversible system.

Equation~(\ref{eq:Alg-DE2}) enforces dynamical equivalence. Equation~(\ref{eq:Alg-CBss2}) can be regarded as a complex balancing condition that uses a different set of rate constants $\alpha_{\vv y_i \to \vv y_j} \kk'_{\vv y_i \to \vv y_j}$. Since $\alpha_{\vv y_i \to \vv y_j} \kk'_{\vv y_i \to \vv y_j} \neq 0$ if and only if $\kk'_{\vv y_i\to \vv y_j} \neq 0$, we preserve the graph structure of $G'_{\vv\kk'}$. It is well-known that a reaction network is weakly reversible if and only if it is complex-balanced for some choice of rate constants~\cite{TDS}. The scaling factor $\vv \alpha$ frees the rate constants from the dynamical equivalence constraint.

Note that while (\ref{eq:Alg-DE2})-(\ref{eq:Alg-scalepos2}) are simple to describe, more sophisticated, computationally efficient methods have been developed~\cite{PolylTimeAlgWR, WRAlgDense}. Weak reversibility is a condition of the underlying directed graph. Ultimately one is imposing conditions on the incidence matrix or the Kirchhoff matrix of the network. Algorithms to find weakly reversible realization for a fixed vertex set have been proposed initially using mixed-integer linear programming~\cite{WRAlgDense, WRAlgSparse} and later by a {\it polynomial time} algorithm based on linear programming~\cite{PolylTimeAlgWR}. However, as with previous work on complex-balanced realizations, one must fix the set of vertices to be used in the computation. According to Theorem~\ref{thm:NoNewNodeWR}, it suffices to find an equivalent network using the existing source vertices. Therefore, the mixed-integer linear programming algorithms proposed in \cite{WRAlgDense, WRAlgSparse} and the polynomial time algorithm in \cite{PolylTimeAlgWR} can be used in conjunction with Theorem~\ref{thm:NoNewNodeWR} to completely characterize whether or not a mass-action system $G_\vv\kk$ is dynamically equivalent to a weakly reversible one.

\section{Conclusion}

If we are looking for a complex-balanced realization of a given polynomial (or power-law) dynamical system, there exists no a priori limit on the number of vertices in the objective network. Moreover, there are no a priori choices for the locations of the vertices. Here we prove that a solution exists if and only if the objective network can be constructed by using only the vertices that are already present in the original system (i.e., the exponents of the monomial terms present in the original system). We also prove that the same is true for detailed-balanced, reversible and weakly reversible systems.

\section{Acknowledgment} 

This work was partially supported by NSF grants DMS-1412643 and DMS-1816238. The work of the third author was partially supported by an NSERC PGS-D award.

\bibliographystyle{siam}
\bibliography{cit}

\end{document}